\documentclass[lettersize,journal]{IEEEtran}
\usepackage{amsmath,amsfonts,amssymb,amsthm}
\usepackage{color,xcolor}
\usepackage{comment}
\usepackage{url}
\usepackage{graphicx}
\usepackage{subcaption}
\usepackage{enumerate}
\usepackage{cite}
\usepackage{mathrsfs}  

\newtheorem{lemma}{Lemma}
\newtheorem{theorem}{Theorem}

\DeclareMathOperator{\e}{\mathrm{e}}
\newtheorem{proposition}{Proposition}
\newtheorem{corollary}{Corollary}

\usepackage{hyperref}%
\hypersetup{colorlinks=true, linkcolor=blue, breaklinks=true, urlcolor=blue, citecolor=blue}

\begin{document}

\title{Design of Cycles by Impulsive Feedback: Application to Discrete Dosing
%\textcolor{magenta}{Application to Dosing of Neuromuscular Blockade Agent}
}

\author{Alexander V. Medvedev, Anton V. Proskurnikov, %~\IEEEmembership{Senior Member}, 
and Zhanybai Zhusubaliyev
        % <-this % stops a space
\thanks{Alexander V. Medvedev [{\tt\small alexander.medvedev@it.uu.se}] is with Department of Information Technology, 
        Uppsala University, SE-752 37 Uppsala, Sweden.
        }%
\thanks{Anton V. Proskurnikov  [{\tt\small anton.p.1982@ieee.org}] is with Department of Electronics and Telecommunications, Politecnico di Torino, Turin, Italy, 10129.}%
\thanks{Zhanybai T. Zhusubaliyev [{\tt\small zhanybai@hotmail.com}] is with Department of Computer Science, International Scientific Laboratory for
Dynamics of Non-Smooth Systems, Southwest State University, Kursk, Russia and Faculty of Mathematics and Information Technology, Osh State University, Lenin st. 331, 723500, Osh, Kyrgyzstan.}
}

%\markboth{Submitted to an IEEE journal,~Vol.~XX, No.~X, XXXX}%
%\markboth{IEEE Transactions on Control Systems Technology,~Vol.~XX, No.~X, XXXX}%
\markboth{IEEE Transactions,~Vol.~XX, No.~X, XXXX}%
{Medvedev \MakeLowercase{\textit{et al.}}: Design of Cycles by Impulsive Feedback}

%\IEEEpubid{0000--0000/00\$00.00~\copyright~2024 IEEE}

\maketitle

\begin{abstract}
The task of maintaining a predefined level of effect in a dynamical plant by applying  periodic control actions often arises in, e.g., process control and medicine. When the state variables of the plant represent the concentrations of chemical substances and the control action constitutes an instantaneous introduction of a certain quantity of a chemical or drug, this control setup is referred to as a (discrete) dosing problem.  The present paper examines an amplitude- and frequency-modulated impulsive controller that, under stationary conditions, generates a desired sequence of uniform and equidistant control impulses based on continuous measurements of the output of a smooth positive nonlinear time-invariant single-input single-output plant with Wiener structure. The controller design method is based on constructing and stabilizing the fixed point of a discrete map that describes the evolution of the state vector of the continuous plant between successive impulsive control action instants. Stability of the fixed point ensures the existence of a basin of attraction around the stationary trajectory, where  solutions of the closed-loop system converge to the stationary solution after perturbation. The convergence rate is determined by the slopes of the amplitude and frequency modulation functions of the impulsive controller. The proposed controller is applied to dosing of the drug \emph{atracurium} in closed-loop  neuromuscular blockade, and its performance is evaluated on a \textcolor{black}{database of patient-specific pharmacokinetic-pharmacodynamic models estimated from clinical data}.%\marginpar{\textcolor{blue}{QAE}} 
It is demonstrated that an implementation of the standard regimen as a pulse-modulated feedback controller  significantly reduces the incidence of underdosing events.
\end{abstract}

\begin{IEEEkeywords}
%TODO (we may need to add predefined keywords during the submission)
Nonlinear dynamical systems, amplitude modulation, frequency modulation, pulse modulation, medical control systems, anesthesia.
\end{IEEEkeywords}

\section{Introduction}
An everyday example of a dosing application is adhering to a doctor's orders on a medication regimen, e.g., ``take one tablet twice a day.'' This open-loop dosing strategy does not account for the actual medication effect in the particular patient. Typically, after initiating the treatment, the doctor evaluates the treated condition's symptoms in the patient and adjusts the regimen accordingly, thereby applying \emph{feedback}. 

%The pharmacotherapy can be intensified in two ways: by increasing the amount of each single dose or by shortening the dosing interval. The choice between them essentially depends on pharmacological properties of the specific drug. Therefore, even manual administration of medication relies on at least a mental model of the drug's pharmacokinetics (PK) and pharmacodynamics (PD). %Another useful observation from the conceptual example above is that the nominal regimen prescribes taking a tablet every twelve hours, which, when followed perfectly, results in a sequence of uniform and equidistant impulsive control actions sustained over an extended period.

%Medication non-adherence poses a challenge in pharmacotherapy, impacting treatment outcomes and being difficult to manage~\cite{AH21}. A common patient concern is how to respond when a regular dose is delayed or missed. Clearly, the PK/PD properties of the drug should be considered in order to properly address the compliance lapse, as both dosing amounts and inter-dose intervals may need adjustment before resuming regular administration. A feedback controller can recommend corrective actions to the patient for timely regimen recovery, provided quantification of the drug effect is available.

Increasing or decreasing the amount of each dose corresponds to the mechanism of amplitude modulation in pulse-modulated control~\cite{GC98,LBS,SP_impulsive}, whereas adjusting the dosing interval represents frequency modulation. The principles of amplitude and frequency modulation feedback implementing discrete dosing are widely employed in biological systems, e.g., in endocrine regulation~\cite{WTT10,MPZ18,TMP19}. 
%A voluminous body of literature describes how these systems can be mathematically modeled and analyzed.

Besides pharmacotherapies, where drugs are administered in tablet or injection form, similar discrete dosing problems -- characterized by intermittent impulsive control actions and continuous effect measurements -- are common in  industrial processes. Examples include space technology, water treatment, food production, chemical and biochemical processes, agriculture, steelmaking and mining. This contrasts with continuous dosing, where the flow rate of a chemical is adjusted to achieve the desired effect. Industrial dosing control systems typically operate in open-loop mode and are implemented using discrete logic or automata~\cite{AH22}. An early example of applying optimal control to dosing can be found in R.~Bellman's work~\cite{B71}. However, open-loop control cannot mitigate the effects of disturbances and model uncertainty in the plant, which necessitates the use of feedback in dosing. %to achieve better performance. 

%Impulsive feedback control is inherently nonlinear. Integrating advanced control laws into closed-loop dynamics adds layers of complexity, making the stability and performance analysis of the resulting system challenging. However, until recently, simple pulse-modulated feedback solutions manipulating the amplitude and frequency of the control impulses have been notably absent. The preferred method for designing impulsive feedback is currently Model Predictive Control (MPC). Motivated by a drug-dosing application and an available PK/PD model, an MPC with impulsive control action was proposed in~\cite{SPS15}. 
%A promising application of impulsive MPC to insulin dosing in simulated diabetes patients is reported in~\cite{RGS20,Estremera2023}.
%The physiological profile of insulin secretion includes around ten major hormone pulses over 24 hours~\cite{PGV88}, whose timing is closely related to meal times. 
%\todo{Emphasize that our solution is biomimetic and inspired by neurohormone impulsive regulation.}

This paper develops and evaluates a new framework for impulsive dosing control design, hinging on the findings in recent conference papers~\cite{MPZ23,MPZ23a,MPZ24}. The latter publications demonstrate that a nonlinear amplitude and frequency pulse modulator can be designed to control a positive continuous linear time-invariant third-order plant to a specified periodic solution. By reducing the hybrid system dynamics to a discrete nonlinear map, the local transient properties of the closed-loop system are determined by the location of the multipliers of the fixed point corresponding to the stationary periodic solution. It is readily observed that the structure of the closed-loop system is identical to that of the Impulsive Goodwin's Oscillator (IGO), a mathematical model of pulsatile endocrine regulation~\cite{MCS06,Aut09,PRM24}. Therefore, the pulse-modulated controller considered further in this paper can be seen as biomimetic.

The application illustrating the utility of the proposed framework is neuromuscular blockade (NMB).
NMB causes skeletal muscle relaxation and is routinely used in anesthesia to optimize surgical conditions. Underdosing NMB can lead to inadequate paralysis, while overdosage may extend neuromuscular block beyond the time necessary for surgery and anesthesia.
The effect of NMB agents is measured by neuromuscular monitors~\cite{MH06}, devices that electrically stimulate a peripheral nerve while also quantifying the evoked responses. 

The long-term (one to ten days) NMB is practiced  during mechanical ventilation in the intensive care unit. It has become especially common in connection with treatment for COVID-19 during inpatient hospitalization~\cite{TZP22}. There is wide interpatient variability in required NMB agent dosage  and the latter may decrease or increase with time. To ensure proper dosing, it is therefore important to monitor the depth of NMB.

%Compared to administration of fixed doses (open-loop control), using the monitors for dose titration (i.e. feedback) significantly reduces exposure to NMB drugs without affecting the observed clinical outcome~\cite{TSB21}.

\textcolor{black}{NMB drugs are administered either in intermittent doses or by continuous infusion. After an initial bolus dose, sequential maintenance doses or a certain infusion rate are needed to sustain anesthesia. 
The optimal mode of NMB administration is under debate~\cite{KHK26}. Prolonged continuous NMB infusions carry a higher risk of drug accumulation, causing excessive paralysis, delayed recovery,  and prolonged neuromuscular weakness. Intermittent boluses consume less total drug and allow serial neurological evaluations, but their main drawback is fluctuating levels of paralysis, leading to periods where the patient may not be adequately blocked. This motivates the development of feedback controllers for discrete NMB dosing.}%\marginpar{\textcolor{blue}{Q19.2,\\QAE}}

Closed-loop control of NMB was addressed early in the development of automatic anesthesia since the plant is single-input single-output and the pharmacokinetics (PK) are uncomplicated. When a patient-specific pharmacokinetic-pharmacodynamic (PK/PD) model is available, the controller design problem is not challenging, particularly for maintaining NMB after the initial bolus dose is administered in open loop. For instance, relay control is reported  to handle closed-loop NMB drug administration effectively, ensuring performance appropriate for surgery~\cite{WGB87}. This highlights an important feature of closed-loop drug delivery: Control performance is irrelevant unless it translates into a clinical effect.  In~\cite{UMB88}, patient variability with respect to NMB drugs is identified as the main challenge in closed-loop administration. This issue is addressed by integrating patient-specific control with the support of online system identification of the nonlinear PK/PD model. The same paper also discusses the limitations of using a fixed PID controller, noting that it fails to deliver adequate performance given the model uncertainty. As shown in~\cite{ZMS15}, besides nonlinear oscillations (limit cycles) typical of PID stabilization control of nonlinear systems, deterministic chaos can arise in the closed-loop NMB at lower concentrations of the anesthetic drug. Based on bifurcation analysis, a systematic approach to online recovery from oscillations  is proposed and evaluated in  simulation in~\cite{MZR19}.

\textcolor{black}{
  Recent publications highlighted the benefits of an impulsive mode of drug delivery in anesthesiology, in contrast with continuous infusion that currently dominates the area. A so-called programmed intermittent bolus (PIB) technique uses automatic pumps to deliver regular boluses of medication. Studies suggest that larger volume, longer inter-dose interval boluses are more effective in analgesia than shorter interval, lower volume, or continuous delivery \cite{HLH20}.  At the same time, PIB is currently a strictly open-loop approach without an inherent individualization mechanism and this paper aims at offering a closed-loop implementation of the intermittent bolus principle. Then, with such a solution, the therapeutical benefits of PIB can be combined with the fundamental features of feedback control, such as lower sensitivity to model uncertainty and suppression of exogenous disturbance.
}%\marginpar{\textcolor{blue}{Q19.3,\\Q20.4\\QAE}}

%\subsection*{Contributions}

The main contribution of the present paper is an analysis of a pulse-modulated controller suggested in~\cite{MPZ23} and its  specialization to a realistic drug dosing problem. An analytic expression for the fixed point corresponding to a desired periodic solution of the closed-loop system is provided in Theorem~\ref{pro:fp}. An orbital stability condition for the periodic solution  in terms of the slopes of the modulation functions is proven in Theorem~\ref{th:stability}. This theorem also gives an upper bound for the achievable convergence rate to the periodic solution under perturbation. When evaluated over a population of Wiener-structured PK/PD NMB models estimated from clinical data, a pulse-modulated dosing controller designed for the population mean model parameters is shown to exhibit acceptable performance and robustness. It also significantly improves the incidence of NMB agent underdosing events compared to open-loop administration.

%\subsection*{The paper organization}

The rest of the paper is organized as follows. Section~\ref{sec:NMB-intro} introduces the mathematical model of NMB employed to analyze the properties of the impulsive feedback controller along with the dataset underlying the numerical experiments in this study. The controller design problem at hand is mathematically formulated in Section~\ref{sec:formulation}. 
The closed-loop dynamics are reduced to a return map and the main fixed-point and stability results are established in Section~\ref{sec:nonlin-dynam}. An evaluation of controller performance and robustness on a realistic cohort of NMB patients is presented in Section~\ref{sec:case-study}. Conclusions are drawn in Section~\ref{sec:conclude}, followed by appendices with the proofs of the theorems and lemmas.

\section{Neuromuscular blockade model}\label{sec:NMB-intro}
A continuous-time Wiener model for NMB with the muscle relaxant {\it atracurium} under general  closed-loop  anesthesia  is introduced in~\cite{SWM12}. The model assumes continuous infusion of the drug and the input $u(t)$   is the administered atracurium rate in $\lbrack \mu\mathrm{g}\ \mathrm{kg}^{-1}\mathrm{min}^{-1} \rbrack $, positive and bounded: $0\le u(t) \le u_{\max}$. % The bolus dose administered to a patient in the induction phase is translated in the identification dataset to an equivalent flow over a short time interval. 
The current NMB level determines the model output $y(t)$ 
$\lbrack \% \rbrack$, which is measured by a train-of-four monitor (a peripheral nerve stimulator). The maximal level of output $y(t)=100\%$ is achieved at the instant when the  NMB is initiated and there is no drug in the bloodstream. 

\subsection{Continuous-time Wiener PK/PD model}

The PK model part is assumed to be linear and time-invariant, with a rational transfer function from the input $u(t)$ to the serum drug concentration $\bar y(t)$ defined as follows
\begin{equation}\label{eq:lin_NMB}
W(s)=\frac{\bar Y(s)}{U(s)}=\frac{v_1 v_2 v_3 \alpha^3}{(s+v_1\alpha)(s+v_2\alpha)(s+v_3\alpha)}.
\end{equation}
Here $\bar Y (s)={\cal L}\{\bar y(t)\}$, $U (s)={\cal L}\{ u(t)\}$, and ${\cal L}\{\cdot \}$ denotes the Laplace transform.  The parameter $0<\alpha\le  0.1$ is patient-specific and estimated from data, whereas the other parameters  in~\eqref{eq:lin_NMB} are fixed: $v_1=1$, $v_2=4$, and $v_3=10$. The pole spectrum of~\eqref{eq:lin_NMB} is scaled linearly with $\alpha$, and the static gain is adjusted to one.
The PD part of the NMB model output is  static and relates the output of~\eqref{eq:lin_NMB} to the effect measured by the monitor through a nonlinear Hill-type function
\begin{equation}\label{eq:nonlin_NMB}
y(t)= %\varphi(\bar y)
\varphi(\bar y(t)),\;\;\text{where}\;\;
\varphi(z)\triangleq\frac{100 C_{50}^\gamma}{ C_{50}^\gamma + {z}^\gamma},
\end{equation}
where $C_{50}=3.2425$ $ \mu \mathrm{ g} \ \mathrm{ml}^{-1} $ is the drug concentration that produces 50\% of the maximum effect, and $0<\gamma\le 10$ is a patient-specific parameter. With  model~\eqref{eq:lin_NMB},~\eqref{eq:nonlin_NMB}, the effect of the NMB agent on the patient is captured by a pair $(\alpha,\gamma)$. 

A state-space realization of Wiener model~\eqref{eq:lin_NMB},~\eqref{eq:nonlin_NMB} is 
\begin{equation}                            \label{eq:1}
\dot{x}(t) =Ax(t)+Bu(t), \ \bar y (t)=Cx(t), \ y(t)=\varphi(\bar y(t)),
\end{equation}
where the coefficients of the linear part are
\begin{equation}                            \label{eq:1+}
\begin{gathered}
A=\begin{bmatrix} -a_1 &0 &0 \\ g_1 & -a_2 &0 \\ 0 &g_2 &-a_3 \end{bmatrix}, B=\begin{bmatrix} 1 \\ 0 \\ 0\end{bmatrix}, C^\top =\begin{bmatrix}0\\0\\1\end{bmatrix},\\
a_1\triangleq v_1\alpha,a_2\triangleq v_2\alpha,a_3\triangleq v_3\alpha,g_1\triangleq v_1\alpha,g_2\triangleq v_2v_3\alpha^2,
\end{gathered}
\end{equation}
and the state variables are $x= [x_1,x_2,x_3]^\top$.

It is readily observed that the matrix $A$ is Hurwitz and Metzler. The asymptotic stability of the linear part in~\eqref{eq:1} aligns with the natural decay of chemical substances over time, and the positivity of $x$ ensures that the state variables can be interpreted as concentrations.  The chain structure of the linear part \textcolor{black}{corresponds to a PK/PD model with sequential compartments, where $x_i$ stands for the drug concentration in the $i$-th compartment.}%\marginpar{\textcolor{blue}{Q19.3}}

%represents the dynamics of the sequential PK/PD model compartments, and the state variables correspond to drug concentrations in each compartment. The function $\varphi(\cdot)$ is smooth, positive, and bounded.

%\section{Data set}
\subsection{Dataset}\label{sec:data}
The dataset used in this study is described in detail in~\cite{SWM12}. It was further employed in~\cite{RMM14} to compare the performance of two recursive parameter estimation techniques on clinical data with respect to the NMB model described in Section~\ref{sec:NMB-intro}.  The dynamics of closed-loop controllers based on this model are investigated in~\cite{ZMS15} and~\cite{MZRS16}. 

The model parameter estimates for 48 patients are illustrated in Fig.~\ref{fig:dist}. The population mean parameter values are $\bar\alpha=0.0374, \bar\gamma=2.6677$. The PK parameter $\alpha$ varies by $48\%$ across the dataset, whereas the PD parameter $\gamma$ varies by nearly $75\%$. 
%Given the nonlinear character of the PD, the latter uncertainty is the most challenging one~\cite{MP26}. 
The correlation between the estimates of $\alpha$ and $\gamma$ is low, as seen in Fig.~\ref{fig:alpha_gamma}. 
However, models with high values of $\gamma$ do not exhibit high values of $\alpha$; in contrast, the model with Patient Identification Number (PIN) 26 features the lowest value of $\gamma$ and the highest value of $\alpha$.

\textcolor{black}{A feasibility analysis of the models in the dataset was performed in \cite{MP26}, where the maintenance dose $\lambda^*$ and dosing period $T^*$ that are necessary to keep the output $y(t)$ within
\begin{equation}\label{eq:clinical_boundaries}
{\bf y}_{\min}=2 \le y(t)\le {\bf y}_{\max}=10,
\end{equation}
were calculated for each model. Models that required exaggerated drug doses (over $\lambda_{\max}=600~\mathrm{\mu g/kg}$) to produce acceptable effect ($y(t)<{\bf y}_{\max}$) were judged infeasible. Some of them combined elevated doses with
reasonable timing whereas other ones exhibited prolonged  effect (longer than $T_{\max}=45~\mathrm{min}$) along with exaggerated maintenance dose. The resulting model classification and parameter clustering are depicted in Fig.~\ref{fig:alpha_gamma}.
All the infeasible models exhibit low values of $\gamma$, i.e., low drug sensitivity. Yet, the actual distinction between feasible and infeasible models is more intricate. }%\marginpar{\textcolor{blue}{Q19.1\\Q20.4}}
\begin{figure}[ht]
\centering 
\includegraphics[width=0.99\linewidth]{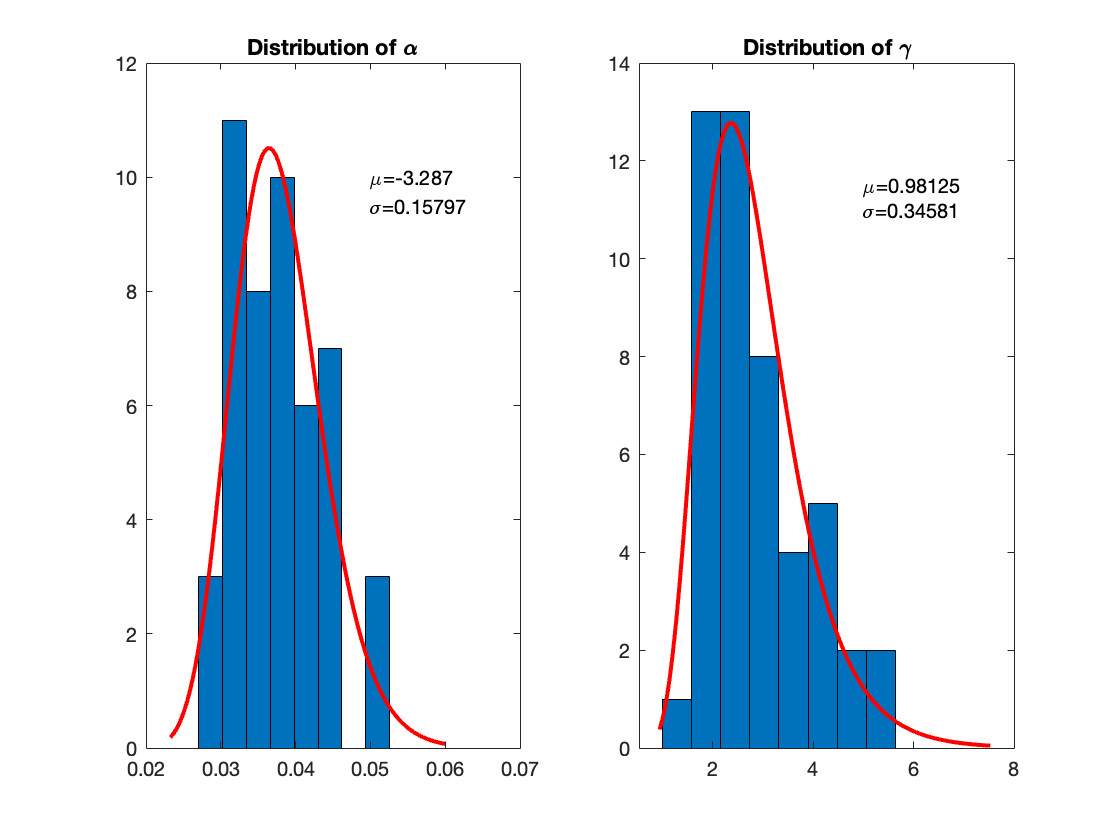}
\caption{Histograms (blue) and estimated  lognormal distributions (red) for the model parameters in the data set. Lognormal distribution is selected due to the positivity of the model parameters $(\alpha,\gamma)$.}\label{fig:dist}
\end{figure}

%\marginpar{\textcolor{blue}{Q19.1}}
\begin{figure}[ht]
\centering 
\includegraphics[width=1.0\linewidth]{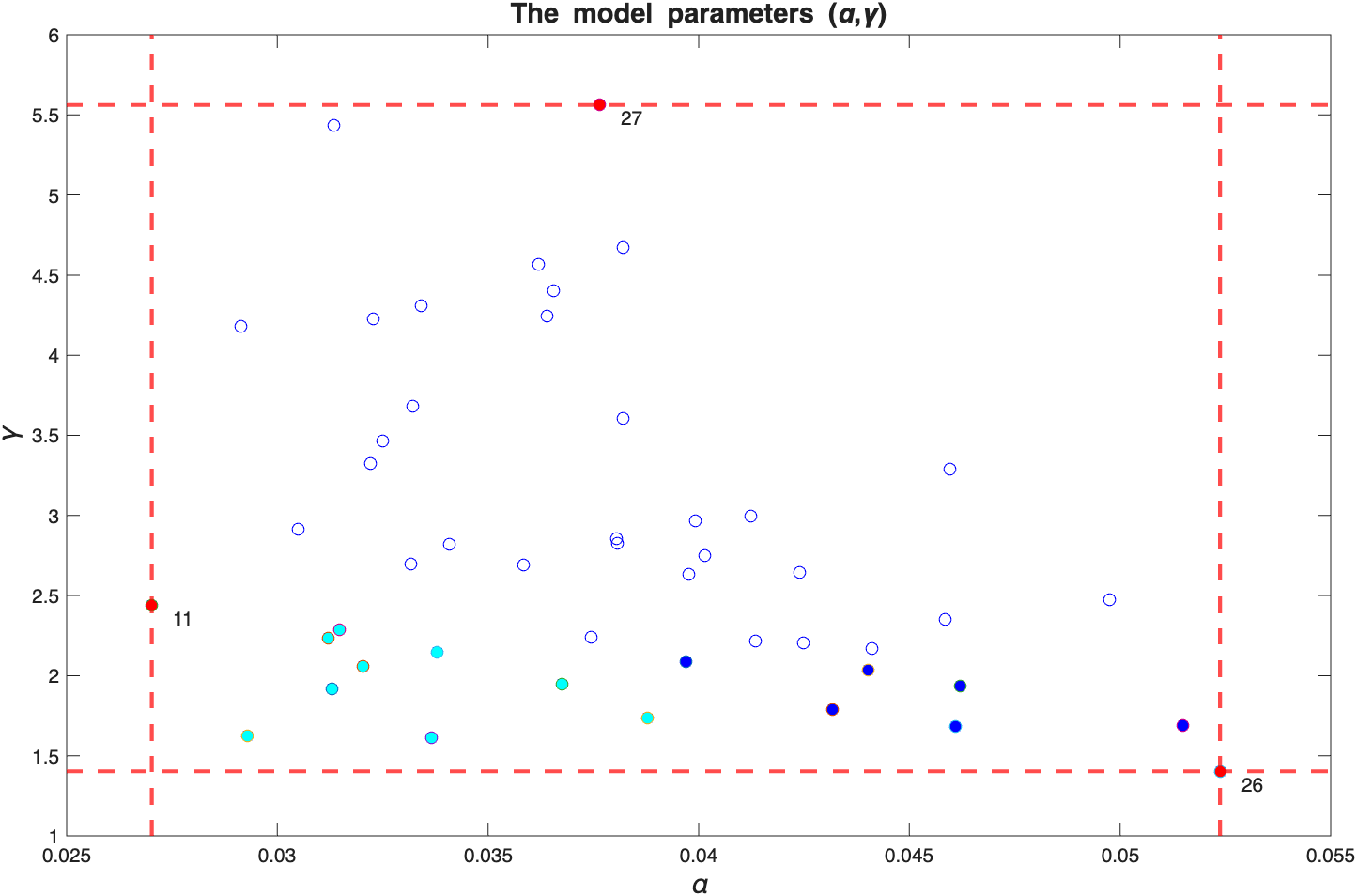}
\caption{The model parameter pairs $(\alpha,\gamma)$ in the dataset: $ \alpha_{\min}= 0.0270\le \alpha \le 0.0524=\alpha_{\max}$, $\gamma_{\min}=1.4030\le\gamma\le 5.5619=\gamma_{\max}$. \textcolor{black}{ The cases exhibiting extreme parameter values are marked with filled red circles and indicated by the Patient Identification Number. Models pointed out as infeasible in \cite{MP26} are marked with filled blue and cyan circles. Models with $\lambda_{\max}<\lambda^*$ and $T^*<T_{\max}$ are plotted in blue, whereas models with $\lambda_{\max}<\lambda^*$ and $T_{\max}<T^*$ are plotted in cyan.}}\label{fig:alpha_gamma}
\end{figure}

\section{Problem formulation}\label{sec:formulation}

%Consider the state-space realization of  a Wiener model  
%\begin{equation}                            \label{eq:1}
%\dot{x}(t) =Ax(t)+Bu(t), \ z(t)=Cx(t), \ y(t)=\psi(z(t)),
%\end{equation}
%where
%$$A=\begin{bmatrix} -a_1 &0 &0 \\ g_1 & -a_2 &0 \\ 0 &g_2 &-a_3 \end{bmatrix}, B=\begin{bmatrix} 1 \\ 0 \\ 0\end{bmatrix}, C^\top %=\begin{bmatrix}0\\0\\1\end{bmatrix}.$$
%Then  NMB model~\eqref{eq:lin_NMB},~\eqref{eq:nonlin_NMB}, is written as~\eqref{eq:1}
%with $a_1=v_1\alpha$, $a_2=v_2\alpha$, $a_3=v_3\alpha$, $g_1=v_1\alpha$, $g_2=v_2v_3\alpha^2$, and the state variables are $x= %[x_1,x_2,x_3]^\top$. 

%The chain structure of the linear part portrays dynamics of three substances' concentrations, where a preceding substance stimulates the production of the next one.

In what follows, the continuous plant in ~\eqref{eq:1} is controlled by an impulsive output feedback, characterized by frequency and amplitude pulse modulation operators~\cite{GC98}. This control action can be described as an infinite sequence of Dirac delta-functions fed as input to the linear block~\eqref{eq:1}, $u(t)=\sum_{n=0}^{\infty}\lambda_n\delta(t-t_n)$, where the weights and firing times of the impulses depend on the nonlinear plant output $y$. Equivalently, the effect of impulses is represented by instantaneous jumps in the state vector, whose timing is governed by a difference equation. This leads to the following hybrid dynamics:
\begin{gather}
\dot x(t)=Ax(t),\;\;t\in(t_n,t_{n+1}),\notag\\ 
\quad  x(t_n^+) = x(t_n^-) +\lambda_n B, \quad x(0^-)=x_0,\label{eq:2}\\                                         
t_{n+1} =t_n+T_n, \quad t_0=0,\notag\\  
T_n =\bar\Phi(y(t_n)), \ \lambda_n=\bar F(y(t_n)),\notag 
\end{gather}
where  $n=0,1,\ldots$ and the output $y(t)$ is defined in~\eqref{eq:1}. 
The minus and plus in a superscript in~\eqref{eq:2} denote the left-sided and
right-sided limits, respectively. The instants $t_n$ are termed (impulse) firing times,
and $\lambda_n$ represents the corresponding impulse weight. 
Despite the jumps in~\eqref{eq:2}, $y(t)$ and $\bar y(t)$ remain continuous, since $\varphi(\cdot)$ is smooth and
 %\begin{equation}\label{CBLB}
 $CB=CAB=0$. Since $CA^2B\ne 0$,
 %\end{equation}
 the linear block is a system of relative degree $3$.
 
The \emph{design degrees of freedom} of the impulsive controller in question are the frequency modulation function $\bar\Phi(\cdot)$ and the amplitude modulation function $\bar F(\cdot)$.
 
With $\circ$ denoting composition,  introduce the functions
\begin{equation}\label{eq.phi_f_nobar} 
\Phi(\cdot)\triangleq (\bar\Phi\circ\varphi)(\cdot), \quad F(\cdot)\triangleq (\bar F\circ\varphi)(\cdot).
\end{equation}

The following restrictions on the design degrees of freedom are imposed.
Both $F(\cdot)$ and $\Phi(\cdot)$ are assumed to be continuous and monotonic, with $F(\cdot)$ being nonincreasing
and $\Phi(\cdot)$ being nondecreasing on $[0,\infty)$. To guarantee boundedness of closed-loop solutions in~\eqref{eq:1},~\eqref{eq:2}, it is required that
\begin{equation}                             \label{eq:2a}
0<\Phi_1\le \Phi(\cdot)\le\Phi_2, \quad 0<F_1\le F(\cdot)\le F_2,
\end{equation}
where $\Phi_1$, $\Phi_2$, $F_1$, $F_2$ are constants.

\textcolor{black}{Under the assumptions made on the modulation functions, impulsive controller~\eqref{eq:2} enforces a negative feedback on the linear plant. When the output $\bar y(t)$ increases, the controller responds with less frequent (longer $T_n$) impulses of lower weight $\lambda_n$, thus reducing the output; when the output decreases, the control response is the opposite one. A mathematical rationale for this property is provided in Section~\ref{par:fb}.}%\marginpar{\textcolor{blue}{QAE}}

\paragraph*{\bf Control Problem} The problem at hand is to select the modulation functions $\bar\Phi(\cdot)$, $\bar F(\cdot)$ so that  closed-loop system~\eqref{eq:2} exhibits  an orbitally stable periodic solution with a predefined period $T>0$ and a given pulse weight 
$\lambda>0$, i.e., a solution with
$\lambda_n\equiv\lambda,\;\;T_n\equiv T$, for all $n$.

\section{Closed-loop dynamics and 1-cycles}\label{sec:nonlin-dynam}

Under the assumptions introduced in the previous section and with the plant nonlinearity $\varphi$ incorporated in the modulation functions $\Phi$ and $F$, closed-loop system~\eqref{eq:2} is identical to the Impulsive Goodwin's Oscillator (IGO)~\cite{MCS06,Aut09}, a hybrid mathematical model originally devised to describe pulsatile endocrine regulation. Denoting $X_n=x(t_n^-)$, the  state vector sequence of the IGO  obeys the impulse-to-impulse (or \emph{return}) map~\cite{Aut09,ZCM12b} as follows:
\begin{align}\label{eq:map}
    X_{n+1}&=Q(X_n),\\
    Q(\xi) &\triangleq \mathrm{e}^{A\Phi(C\xi)}\left( \xi+ F(C\xi)B \right).\nonumber
\end{align}
Between the firing instants, the continuous state trajectory on the interval $(t_n,t_{n+1})$ is uniquely defined by $X_n$ as
\begin{equation} \label{eq:1d}
x(t)=\e^{(t-t_n)A}(X_n+\lambda_n B),\quad t\in(t_n,t_{n+1}).
\end{equation}

\subsection{The fixed point and 1-cycle}

As shown in~\cite{Aut09}, the mapping $Q$ has a unique fixed point
\begin{equation}\label{eq:1-cycle}
    X=Q(X),
\end{equation}
for every pair of the nonlinear functions $F,\Phi$ in~\eqref{eq.phi_f_nobar} that are, respectively, nonincreasing and nondecreasing, and 
obey~\eqref{eq:2a}. This fixed point determines a special type of periodic solution, termed \emph{1-cycle}~\cite{ZM03,ZCM12b} and 
characterized by only one firing of the feedback in the (least) period. Denoting $\bar y_0\triangleq CX$, the 1-cycle is uniquely determined by the solution of~\eqref{eq:1-cycle} and obtained by substituting $x(t_n^-)=X_n=X$, $T_n=T\triangleq\Phi(\bar y_0)$, and
$\lambda_n=\lambda\triangleq F(\bar y_0)$ into~\eqref{eq:1d} (then, obviously, $X_{n+1}=Q(X)=X=X_n$, so that the hybrid solution is $T$-periodic).
 
Using the Opitz formula~\cite{E87,DeBoor2005} from matrix calculus, the solution of~\eqref{eq:1-cycle} for the given values of $\lambda$ and $T$ can be found analytically by using  divided differences. The first divided difference of a function $h(\cdot)$ is defined as
\[
h\lbrack x_1, x_2 \rbrack \triangleq \frac{h(x_1)-h(x_2)}{x_1-x_2},
\]
and higher-order divided differences are defined recursively by 
\[
h\lbrack x_0, \dots, x_k \rbrack = \frac{h\lbrack x_1, \dots, x_k\rbrack-h\lbrack x_0, \dots, x_{k-1}\rbrack }{x_k-x_0}.
\]
Here, for simplicity, only distinct points $x_i\ne x_j$ are considered; a general definition can be found in~\cite{DeBoor2005}.

A closed-form expression of the desired fixed point is provided in the next theorem.
\begin{theorem}\label{pro:fp}
Let $\lambda,T>0$ be fixed and $\mu(x)\triangleq\frac{1}{\e^{-x}-1}$. Then, the following statements are equivalent:
    \begin{enumerate}        
    \item $X$ and $\bar y_0=CX$ obey the equations
        \begin{gather}
        X= \lambda  \begin{pmatrix}
            \mu(-a_1T) \\ g_1{\color{black}T}\mu\lbrack -a_1T,-a_2T \rbrack \\ g_1g_2{\color{black}T^2}  \mu\lbrack -a_1T, -a_2T, -a_3T \rbrack 
        \end{pmatrix},\label{eq:fp_alpha}\\
        \Phi(\bar y_0)=T,\quad F(\bar y_0)=\lambda\label{eq:f-phi-correspondence},
        \end{gather}
    where $a_i,g_i$ are defined in~\eqref{eq:1+}.
    \item $X>0$ is the fixed point of the map $Q(\cdot)$ that corresponds to 1-cycle of the period $T$ and with the pulse weight $\lambda$.
    \end{enumerate}    
\end{theorem}
\begin{proof}
    See Appendix~\ref{app.thm1}.
\end{proof}
\vskip0.2cm
Substituting the expressions for $a_i,g_i$ from~\eqref{eq:1+}, Fig.~\ref{fig:fp_alpha} illustrates the dependence of the fixed point $X$ in~\eqref{eq:fp_alpha} on $\alpha$ for a given pair of $(\lambda,T)$. It can be shown that for each $T>0$ the components of $X$ are decreasing functions of $\alpha$.
\begin{figure}[ht]
\centering 
\includegraphics[width=1.0\linewidth]{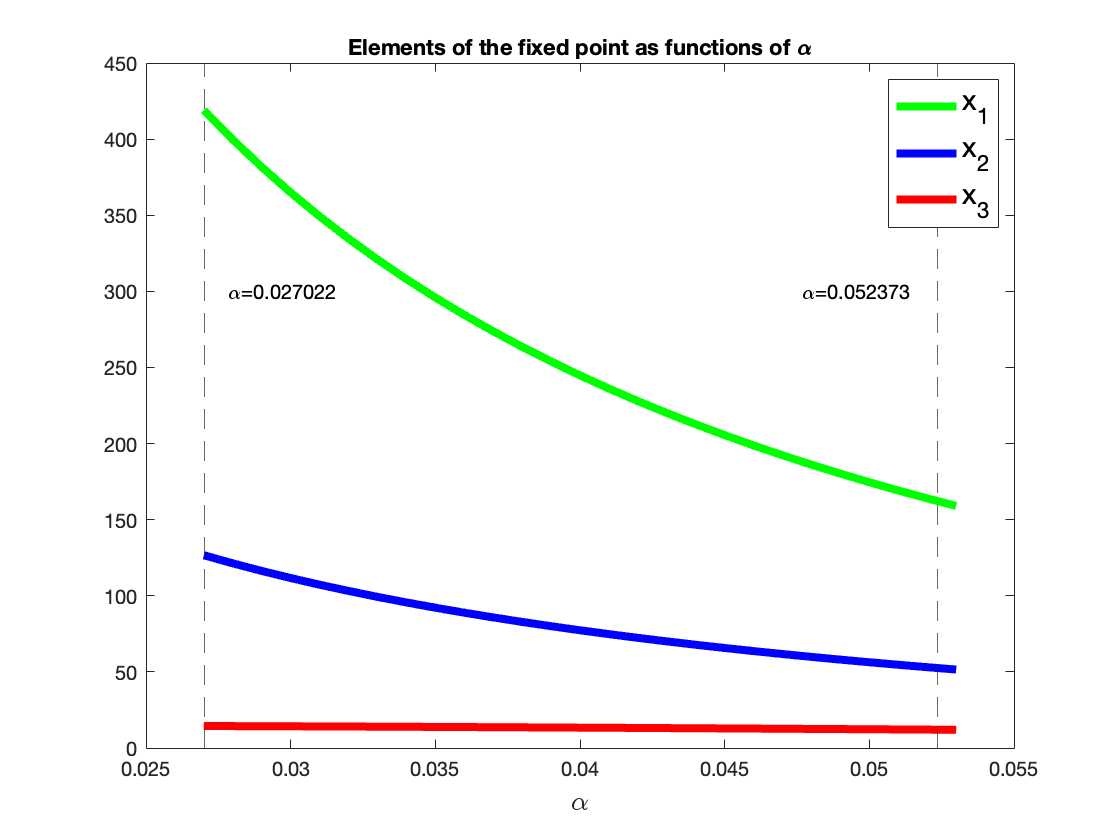}
\caption{Fixed point coordinates as functions of $\alpha\in \lbrack \alpha_{\min}, \alpha_{\max} \rbrack$ {for $\lambda=200$ and $T=20$}. 
%All the elements are decreasing functions of $\alpha$.
}\label{fig:fp_alpha}
\end{figure}

Theorem~\ref{pro:fp} suggests a method for designing the nonlinearities $F$ and $\Phi$ in such a way as to guarantee the existence of a 1-cycle with predefined parameters $\lambda,T$ for a given value of $\alpha$. It suffices to calculate the fixed point $X$ according to~\eqref{eq:fp_alpha}, and then find $F,\Phi$ in such a way that~\eqref{eq:f-phi-correspondence} holds. However, ensuring a sustained 1-cycle requires orbital stability. A complete controller design algorithm is provided in Section~\ref{sec:design}.

\subsection{The Jacobian of the return map at the fixed point}

To sustain the desired  periodic solution in  closed-loop system~\eqref{eq:1},~\eqref{eq:2}, the 1-cycle has to be orbitally stable, which property is guaranteed by stability of the corresponding fixed point. To test stability of the 1-cycle,  the Jacobian of the pointwise map $Q(\cdot)$ is evaluated at the fixed point $X$. 
A straightforward computation shows~\cite{MPZ23} that
    \begin{equation}\label{eq:Q_prime_affine}
        \begin{gathered}
        Q^\prime(X)= \e^{A\Phi(\bar y_0)}+ \left( F^\prime(\bar y_0)J+\Phi^\prime(\bar y_0)D\right)C,\\
        \,D\triangleq AX=F(\bar y_0)A(\e^{-A\Phi(\bar y_0)}-I)^{-1}B,\\
        J\triangleq\e^{A\Phi(\bar y_0)}B.
        \end{gathered}
    \end{equation}
The 1-cycle corresponding to the fixed point $X$ is orbitally stable when the matrix $Q^\prime(X)$ is Schur.

Now the control problem defined in Section~\ref{sec:formulation} can be reformulated in terms of  fixed point.

\paragraph*{\bf Control Problem (Reformulation)} 
Given  plant~\eqref{eq:1} and the desired parameters of the periodic solution, $\lambda$ and $T$, find the modulation functions $\bar\Phi(\cdot)$ and $\bar F(\cdot)$ such that  fixed point~\eqref{eq:fp_alpha} solves~\eqref{eq:f-phi-correspondence} 
and renders $Q^\prime(X)$ Schur-stable.
\vskip0.1cm

\paragraph{Trivial Solution: Periodic Open-Loop Control}
It can be noted that stability of the 1-cycle can be achieved without  impulsive feedback~\eqref{eq:2}, by choosing constant (i.e. output-independent) modulation functions $F(\cdot) \equiv \lambda$ and $\Phi(\cdot) \equiv T$. In this case, $Q'(X) = \e^{TA}$ by~\eqref{eq:Q_prime_affine}. Since $A$ is Hurwitz, $\e^{TA}$ is Schur stable, with spectral radius $\rho(\e^{AT}) = \e^{-a_1 T} = \e^{-\alpha v_1 T}$, using $0 < a_1 < a_2 < a_3$.
This configuration effectively drives~\eqref{eq:1} in open-loop mode with a train of equidistant impulses of constant weight.  

However, impulsive feedback plays a crucial role in enhancing convergence to the desired periodic solution in the presence of deviations. In particular, nonzero slopes $F'(\bar y_0)$ and $\Phi'(\bar y_0)$ allow the spectral radius $\rho(Q'(X))$ to be reduced.

\paragraph{Relation to Static Output Stabilization}

Notice that selecting $F^\prime(\bar y_0)$ and $\Phi^\prime(\bar y_0)$ to stabilize the fixed point $X$ is equivalent to finding a gain $K$ that renders the matrix
\begin{equation}\label{eq:jacobian}
Q^\prime(X) = A_\Phi + W K C,
\end{equation}
Schur stable, where we denote
\[
A_\Phi\triangleq\e^{A\Phi(\bar y_0)}, W\triangleq\begin{bmatrix}
    J &D
\end{bmatrix}, K^\top\triangleq \begin{bmatrix}
  F^\prime(\bar y_0)   & \Phi^\prime(\bar y_0)
\end{bmatrix}.
\]
This problem corresponds to stabilization of the system
\begin{align*}
    x_d(t+1)&=A_d x_d(t)+ B_d u_d(t),\\
    y_d(t)&= C_dx_d(t),  
\end{align*}
by the output feedback $u_d(t)= K_d y_d(t)$. The closed-loop system is stable when the matrix $A_d+B_d K_d C_d$ is Schur. 
The problem of finding such a matrix $K_d$ is known as (static) output feedback stabilization; see~\cite{CLS98} for a review. 
An additional constraint imposed by the impulsive Goodwin oscillator structure is that the elements of the matrix $K$ are sign-definite,  $K_1 \le 0$ and $K_2 \ge 0$, in view of the monotonicity properties of $F$ and $\Phi$.

It is known that the general static output feedback stabilization problem reduces to the feasibility of a nonconvex quadratic matrix inequality~\cite{CLS98}. The special structure of the three-dimensional system considered here allows an \emph{analytic} characterization of all pairs of slope tangents $(\Phi'(\bar y_0), F'(\bar y_0))$ that ensure orbital stability of the 1-cycle corresponding to the parameters $F(\bar y_0) = \lambda$ and $\Phi(\bar y_0) = T$. Such a characterization is provided by Theorem~\ref{th:stability} in the next section.

\paragraph{Impulsive Control as a Negative Feedback}\label{par:fb}

%\subsubsection*{Pulse-modulated control as a (local) negative feedback}

The stability criterion derived in the next section is based on the following property of the pulse-modulated feedback~\eqref{eq:2}, which is of independent interest.
\begin{lemma}\label{lem.signs_DJ}
The vectors $D$ and $J$ in~\eqref{eq:Q_prime_affine} obey the inequalities\footnote{All inequalities involving vectors are understood elementwise.}
\[
D<0\;\;\text{and}\;\; J>0.
\]
\end{lemma}
\begin{proof}
    See Appendix~\ref{app.lemma-proof}.
\end{proof}

Together with the monotonicity assumptions on the modulation functions, Lemma~\ref{lem.signs_DJ} entails
\begin{equation}\label{eq:negative_feedback}
    JF^\prime(\bar y_0) + D\Phi^\prime(\bar y_0) <0,
\end{equation}
for all values of the tangent slopes $F^\prime(\bar y_0)\leq 0$ and $\Phi^\prime(\bar y_0)\geq 0$, one of which is nonzero. Recalling the analogy with the output feedback stabilization problem, inequality~\eqref{eq:negative_feedback} highlights the role of the pulse-modulated feedback~\eqref{eq:2} as negative feedback with respect to the linear block output $\bar y(t)$. This is a principal property of  closed-loop system~\eqref{eq:1},~\eqref{eq:2}, since all signals involved are positive, while the negative feedback principle is implemented in a well-defined mathematical sense. In view of the underlying endocrine regulation problem modeled by the IGO, the impulsive control approach considered here provides a formal explanation of how biofeedback operates. 

From~\eqref{eq:Q_prime_affine} and~\eqref{eq:negative_feedback}, it also follows that increasing $\Phi^\prime(\bar y_0)$ and decreasing $F^\prime(\bar y_0)$ may eventually lead to loss of stability, as some eigenvalues of $Q^\prime(X)$ increase in absolute value. While no analytic expression for the spectral radius of $Q'(X)$ as a function of the tangent slopes is available,  numerical analysis is provided in Subsection~\ref{subsec.rho}.

\subsection{The Schur Stability Criterion}\label{subsec.stab}

Although the expression for the spectral radius of the Jacobian matrix~\eqref{eq:Q_prime_affine} is complicated, an analytic necessary and sufficient condition on the slopes of the modulation functions that ensures $Q'(X)$ is Schur stable does exist. Unlike the standard Jury and Schur–Cohn criteria for stability of a general third-order polynomial, the stability conditions derived below are \emph{linear} with respect to the slope coefficient $F^\prime(\bar y_0)$.
Due to the structure of the Jacobian in~\eqref{eq:Q_prime_affine}, consider the following. 
\paragraph*{\bf Stability Problem} Find all real pairs $(\xi\leq 0,\eta\geq 0)$ such that $\mathscr{Q}(\xi,\eta)$ is Schur stable, where
\begin{equation}\label{eq:Q_general}
        \begin{gathered}
        \mathscr{Q}(\xi,\eta)\triangleq\e^{AT}+ \left( \xi J+\eta D\right)C.       
        \end{gathered}
    \end{equation}

To formulate stability conditions, introduce the functions
\begin{align}
%\chi(s|\xi,\eta)&\triangleq\det\left(sI-\mathscr{Q}(\xi,\eta)\right),\label{eq.chi}\\
\psi(s|\xi,\eta)&\triangleq\frac{\det\left(sI-\mathscr{Q}(\xi,\eta)\right)}{\det(sI-\e^{AT})}\notag\\
&=1-C(sI-\e^{AT})^{-1}(\xi J+\eta D),\label{eq.psi}\\
c(\eta)\triangleq&\e^{-(a_1+a_2+a_3)T}\left(1+\eta\lambda CA(I-\e^{AT})^{-1}B\right).\label{eq.c0}
\end{align}
Notice that the poles of $\psi(\cdot|\xi,\eta)$ are the real numbers $s=\e^{-a_iT}$, $i=1,2,3$, according to~\eqref{eq:1+}. 
In the equations to follow, dependence on $(\xi,\eta)$ in $\mathscr{Q}$ and $\psi$ is dropped for brevity when it does not lead to confusion.
Using the definitions of $J$ and $D$ in~\eqref{eq:Q_prime_affine}, it can be shown that $\psi(0|\xi,\eta)$ does not depend on $\xi$, since $C\e^{-AT}J=CB=0$. Also, $c(\eta)=\psi(0|\xi,\eta)\e^{-(a_1+a_2+a_3)T}$ for all pairs $(\xi,\eta)$.

\begin{theorem}\label{th:stability}
Let $\xi\leq 0$ and $\eta\geq 0$ be arbitrary real constants.
\textbf{Non-critical case:} 
If $c(\eta)\ne 0$, then $\mathscr{Q}=\mathscr{Q}(\xi,\eta)$ is Schur stable if and only if the three inequalities hold as follows:
\begin{gather}
c(\eta)>-1,%=\psi(0|\eta,\xi)\e^{(a_1+a_2+a_3)T}>-1,
\label{eq.necess-stab1}\\
%\psi(0)=1+\eta\lambda CA(I-\e^{AT})^{-1}B>-\e^{(a_1+a_2+a_3)T},\label{eq.necess-stab1}\\
\psi(-1)=1+C(I+\e^{AT})^{-1}\left(\xi J+\eta D\right)>0,\label{eq.necess-stab2}\\
c(\eta)\psi(c(\eta)|\xi,\eta)>0.\label{eq.necess-stab3}
\end{gather}
In this case, the spectral radius of $\mathscr{Q}$ is greater than $|c(\eta)|$.

\textbf{Critical case:} In the case where $c(\eta)=0$, 
$\mathscr{Q}(\xi,\eta)$ is Schur stable if and only if~\eqref{eq.necess-stab2} holds and 
\begin{equation}\label{eq.necess-stab3+}
|C\e^{-2AT}(\xi J+\eta D)|<\e^{(a_1+a_2+a_3)T}.
\end{equation}
\end{theorem}
\begin{proof}
    See Appendix~\ref{app.thm2}. 
\end{proof}
Notice that the left-hand side of~\eqref{eq.necess-stab3} is \emph{linear} in $\xi$ and rational in $\eta$. Hence, unlike the standard Schur stability conditions for matrices~\cite{MPZ23a,FGL98}, the conditions of Schur stability for  Jacobian~\eqref{eq:Q_prime_affine} prove to be \emph{linear} in $\xi$. Similarly,~\eqref{eq.necess-stab3+} reduces to two inequalities that are linear in $\xi$.
%\textcolor{olive}{This enables the possibility to represent the area $(\xi,\eta)$ graphically as shown in Fig.~\ref{fig:stability_amplitude_frequency}.}

The following corollary is straightforward by noticing that \[
Q'(X)=\mathscr{Q}(F'(\bar y_0),\Phi'(\bar y_0)).
\]
\begin{corollary}\label{col:stab_F_Phi}
The Jacobian~\eqref{eq:jacobian} corresponding to a 1-cycle is Schur stable if and only if $\xi=F'(\bar y_0)$ and
$\eta=\Phi'(\bar y_0)$ satisfy the stability conditions of Theorem~\ref{th:stability}. 
\end{corollary}

\subsection{Local Convergence Rate}\label{subsec.rho}

The eigenvalues of the Jacobian $Q'(X)$ (that is, the multipliers of the fixed point) define the convergence rate to the periodic solution and the dynamical character of the transients in vicinity of the fixed point. In dosing applications, to avoid overdosing, it is desirable to achieve monotone convergence to the steady-state conditions. With respect to (local) convergence to the desired 1-cycle, it is then required that the fixed point has to possess positive  multipliers. The property can be demonstrated by diagonalizing the Jacobian via a state transformation since it describes the linearized dynamics of the discrete map $Q(\cdot)$ in~\eqref{eq:map}.
% \textcolor{red}{Looks like a conjecture based on simulations. May be, insert some plot here? Also, in some conference paper there was a picture illustrating a big overshoot in the case of large spectral radius, may be, put it here as a motivation to minimize it? In general, we need to explain that the original problem is nonlinear, so the spectral radius is only one of the characteristics.}

By making use of~\eqref{eq:Q_prime_affine},
the problem of minimizing the spectral radius  of the Jacobian by selecting $K$ in~\eqref{eq:jacobian} is
\[
K^*=\arg\min_K \rho( A_\Phi+WKC),
\]
where $\rho(\cdot)$ denotes the spectral radius. The problem of numerically minimizing the spectral radius of a nonsymmetric affine matrix function  is considered in, e.g.,~\cite{OW88}. It is both nonconvex and nonsmooth as the eigenvalues are generally not differentiable. Theorem~\ref{th:stability} implies that the spectral radius cannot be reduced below $|c(\Phi'(\bar y_0))|$. Analyzing  definition~\eqref{eq.c0}, one notices that the factor $\e^{-(a_1+a_2+a_3)T}=\e^{-\alpha (v_1+v_2+v_3)T}$ is the larger, the smaller $\alpha>0$ is. \textcolor{black}{Therefore, patient models with smaller values of $\alpha$ will be more challenging to control to the desired periodic solution than those with higher values of the parameter.} However, it can be shown (Corollary~\ref{cor.rho-convex} in the Appendix) that $CA(I-\e^{AT})^{-1}B<0$; thus, by choosing $\Phi'(\bar y_0)>0$ large, one can decrease the value of $|c(\eta)|$ (and, as shown numerically below, as well the spectral radius).

\paragraph{Amplitude modulation} To obtain a better insight into the spectral properties of $Q^\prime(\cdot)$ and how they depend on the impulsive controller, consider a special case of amplitude modulation that is obtained from~\eqref{eq:2} by letting $\Phi(z)\equiv T$. Then the Jacobian takes the form of
\begin{equation}\label{eq:jacobian-amplitude}
Q^\prime_F(X)=\e^{A\Phi(\bar y_0)}+F^\prime(\bar y_0)JC=\e^{AT}(I+F'(\bar y_0)BC).
\end{equation}
Stability condition~\eqref{eq.necess-stab1} is automatically satisfied when $\eta=\Phi'(\bar y_0)=0$. Furthermore, $c(0)=\e^{-(a_1+a_2+a_3)T}$ is  very small for reasonable choices of $T$ and, for this reason, 
\[
\psi(c(0)|\xi,0)\approx\psi(0|\xi,0)=1+\xi C\e^{-TA}J=1+\xi CB=1>0,
\]
so that~\eqref{eq.necess-stab3} typically holds unless $|F'(\bar y_0)|$ is very large. 
The most restrictive assumption is~\eqref{eq.necess-stab2}, which requires that
\[
\frac{-1}{C(I+\e^{AT})^{-1}J}<\xi=F'(\bar y_0)<0.
\]

The coefficients of the characteristic polynomial (see Proposition~\ref{pr:amplitude_modulation} in the Appendix~\ref{app:E}) provide information on the eigenvalues $s_1,s_2,s_3$, since $\gamma_1=\mathrm{Tr}~Q^\prime_F(X)=\sum_{i=1}^3 s_i$ and $\gamma_3=\det Q^\prime_F(X)=\prod_{i=1}^3 s_i$. Apparently, the product of the eigenvalues of the Jacobian is independent of the amplitude modulation feedback, and a decrease in one of the eigenvalues will be accompanied with a rise in the other ones. The sum of the eigenvalues is an affine function of $F^\prime(\bar y_0)$. 

%\textcolor{red}{In all simulations: what was $T$?}

To maximize the convergence rate to the desired solution of the dynamics linearized at the fixed point of  closed-loop system~\eqref{eq:1},~\eqref{eq:2}, one seeks the value of $F^\prime(\bar y_0)$ that minimizes the spectral radius of  $Q^\prime_F(X)$. Fig.~\ref{fig:abs_eig_F} shows the absolute values of the eigenvalues of $Q^\prime_F(X)$ as a function of $F^\prime(\bar y_0)$. Interestingly,  the spectral radius is minimized at the \emph{double multiplier} point, where two real multipliers of $Q^\prime_F(X)$  
merge and then split into a complex-conjugate pair. As shown in Appendix~\ref{app:E}, in this case the eigenvalues are found from the coefficients of the characteristic polynomial, so the spectral radius (conjectured to be minimal) can be evaluated explicitly.

\begin{figure}[ht]
\centering 
\includegraphics[width=0.7\linewidth]{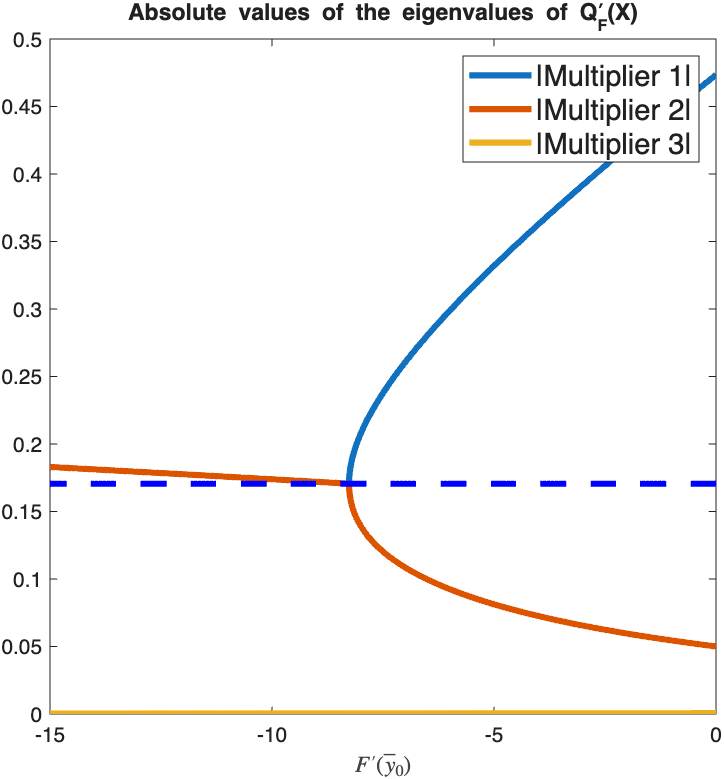}
\caption{Absolute values of the eigenvalues of $Q^\prime_F(X)$ as function of $F^\prime(\bar y_0)$. The minimal spectral radius is depicted by dashed line. The absolute values of two eigenvalues (a complex pair) coincide after the double multiplier point. The third eigenvalue (yellow line) is small in comparison with the other two and slowly decreases with the decreasing $F^\prime(\bar y_0)$. %\textcolor{red}{Axis legend wrong: $z_0$ must be $\bar y_0$}
}\label{fig:abs_eig_F}
\end{figure}
\begin{figure}[ht]
\centering 
\includegraphics[width=0.9\linewidth]{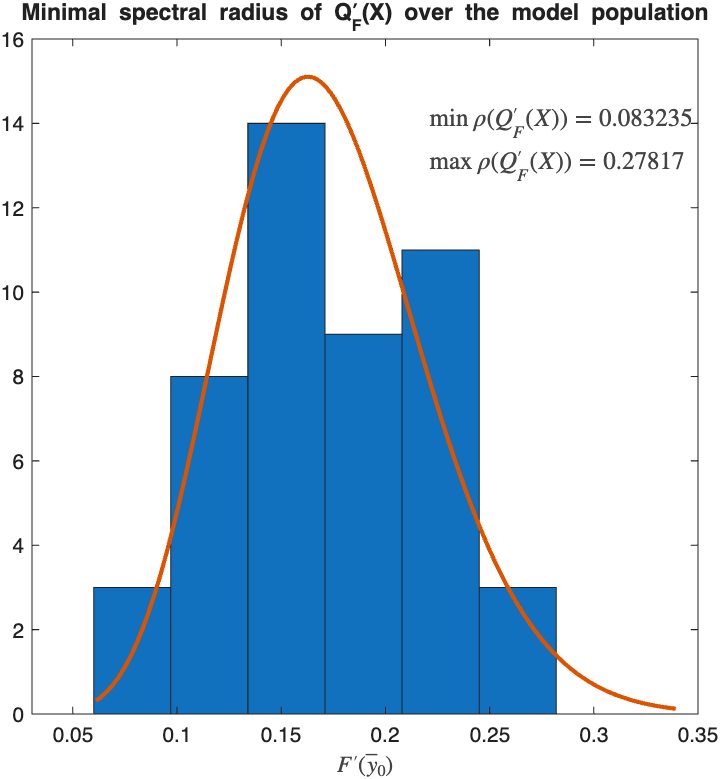}
\caption{Histogram of minimal spectral radii of $Q^\prime_F(X)$ over the NMB model population. An approximation with a Beta distribution is provided for reference. 
%\textcolor{red}{Add $\rho$ below the axis?}
}\label{fig:hist_rho_F}
\end{figure}
\begin{figure}[ht]
\centering 
\includegraphics[width=1.0\linewidth]{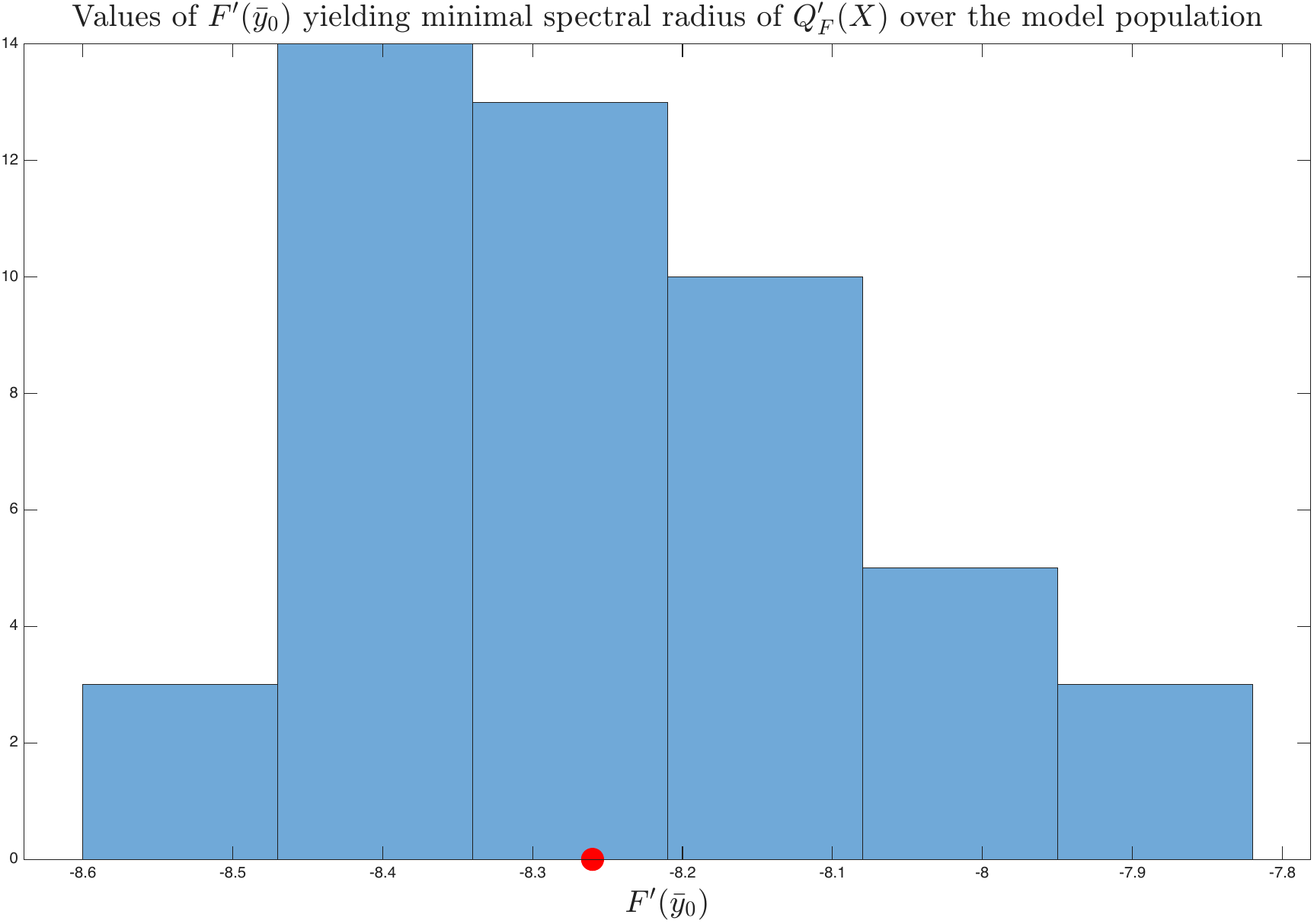}
\caption{Histogram of values of $F^\prime(\bar y_0)$ yielding minimal spectral radius of $Q^\prime_F(X)$ over the NMB model population. The optimal value for the population mean  of $F^\prime(\bar y_0)=-8.2600$  is marked with red circle. 
%\textcolor{red}{Wrong legend: $z_0$ must be $\bar y_0$, title above and the axis.} 
}\label{fig:hist_opt_K_F}
\end{figure}

The minimal spectral radius of $Q^\prime_F(X)$ varies significantly over the considered model population, Fig.~\ref{fig:hist_rho_F}. Notably, this does not result in a wide spread of the optimal values of $F^\prime(\bar y_0)$ at which the minimal spectral radius is achieved. The optimal value of $F^\prime(\bar y_0)$  calculated for the population mean value of $\alpha$ appears approximately in the middle of the interval, see Fig.~\ref{fig:hist_rho_F}.

%\textcolor{blue}{Let $Q^\prime_F(X)$ have a complex eigenvalue $\xi=a+ib$. Then
%\begin{align*}
%    \mathrm{Re} {\cal D}(\xi) &= a^3-\gamma_1 a^2- (3b^2+\gamma_2)a +b^2\gamma_1-\gamma_3, \\
%    \mathrm{Im} {\cal D}(\xi) &= 3a^2-2\gamma_1 a-\gamma_2-b^2.
%\end{align*}
%Now
%\[
%\frac{{\mathrm{d}}{\cal D}(s)}{\mathrm{d} s}=3s^2-2\gamma_1 s-\gamma_2,
%\]
%and 
%\[
%\mathrm{Im} {\cal D}(\xi)=\frac{{\mathrm{d}}{\cal D}(s)}{\mathrm{d} s}\biggr\rvert_{s=a}-b^2.
%\]
%}

\paragraph{Frequency modulation} Similarly to the case of pure amplitude modulation above, consider 
\[
Q^\prime_\Phi(X)=\e^{AT}+\Phi^\prime(\bar y_0)DC.
\]
This type of impulsive feedback is obtained from~\eqref{eq:2} by assuming $F(\bar y_0)\equiv \lambda$. The slope of the frequency modulation function has to be positive in order to  enforce sparser drug administration intervals for an elevated output. 
Taking into account the expression for $D$ in~\eqref{eq:Q_prime_affine},
the Jacobian is found as
\[
Q^\prime_\Phi(X)=\e^{AT}{(I-\e^{AT})}^{-1}(I-\e^{AT}+\lambda \Phi^\prime(\bar y_0) ABC).
\]

Similar to the case of amplitude modulation, the fastest (local) convergence to the 1-cycle (i.e. the smallest spectral radius of $Q^\prime_\Phi$) is achieved at the double multiplier point, where the multipliers become complex; see Fig.~\ref{fig:abs_eig_Phi}.
\begin{figure}[ht]
\centering 
\includegraphics[width=0.9\linewidth]{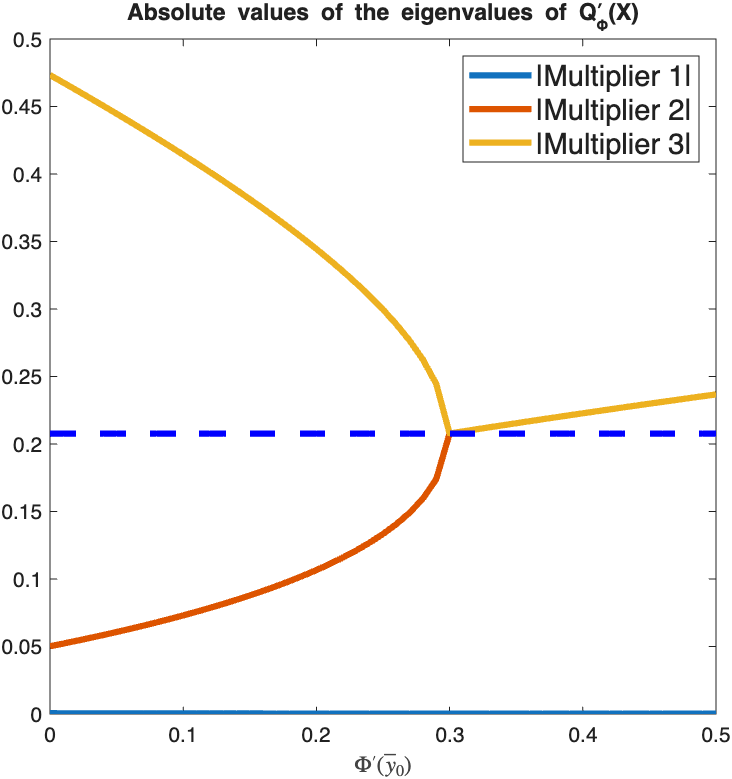}
\caption{Absolute values of the eigenvalues of $Q^\prime_\Phi(X)$ as a function of $\Phi^\prime(\bar y_0)$. The population mean value is assumed for $\alpha$. The minimal spectral radius is depicted by the dashed line. %\textcolor{red}{$z_0$ to be corrected}  
}\label{fig:abs_eig_Phi}
\end{figure}

\begin{figure}[ht]
\centering 
\includegraphics[width=0.7\linewidth]{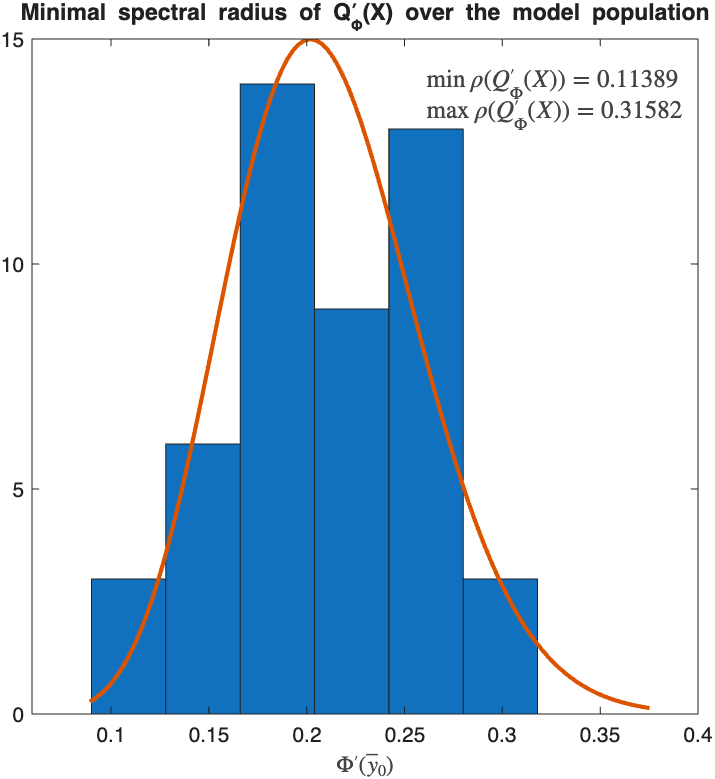}
\caption{Histogram of minimal spectral radii of $Q^\prime_\Phi(X)$ over the NMB model population. An approximation with a Beta distribution is provided for reference. 
%\textcolor{red}{$z_0$ to be corrected}
}\label{fig:hist_rho_Phi}
\end{figure}

\begin{figure}[ht]
\centering 
\includegraphics[width=0.7\linewidth]{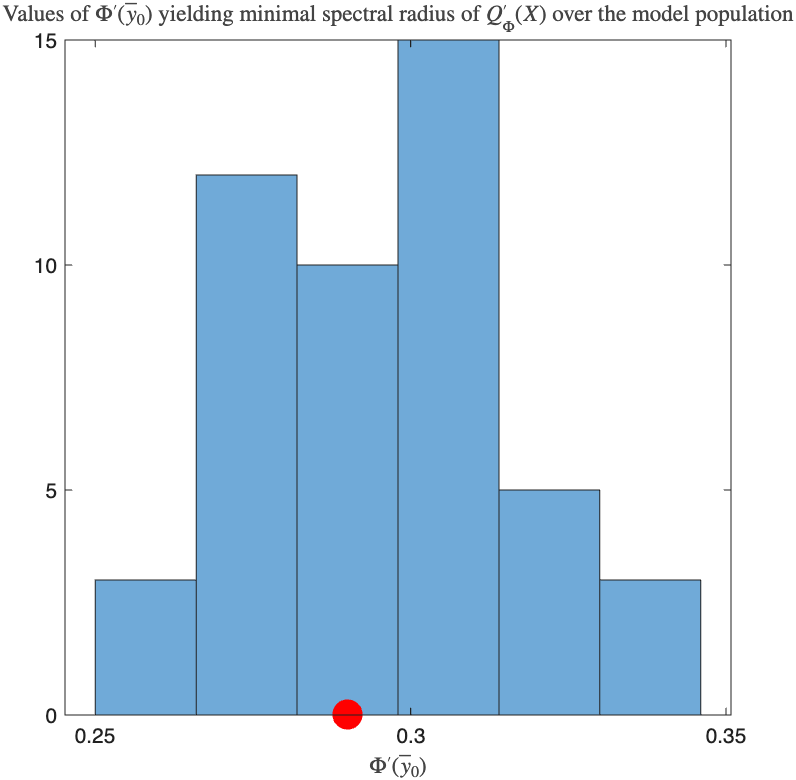}
\caption{Histogram of values of $\Phi^\prime(\bar y_0)$ yielding minimal spectral radius of $Q^\prime_\Phi(X)$ over the NMB model population. The optimal value for the population mean  of $\Phi^\prime(\bar y_0)=0.29$  is marked with red circle. %\textcolor{red}{Wrong legend under the axis, must be $\Phi$, not $F$ and $z_0$ to be replaced.}
}\label{fig:hist_opt_K_Phi}
\end{figure}

Over the model population, the achievable convergence rate values with respect to the frequency modulation feedback in Fig.~\ref{fig:hist_rho_Phi} are distributed similarly to the case of amplitude modulation in Fig.~\ref{fig:hist_rho_F}. However, the required slopes of the frequency modulation characteristic are much lower; see Fig.~\ref{fig:hist_opt_K_Phi}. Once again, the value of $\Phi^\prime(\bar y_0)$ corresponding to the population mean value of $\alpha$ is approximately in the middle of the range.

\paragraph{Amplitude and frequency modulation}

The convergence to the desired 1-cycle, when both amplitude and frequency modulation are exploited in  closed-loop system~\eqref{eq:1},~\eqref{eq:2}, is difficult to analyze analytically. In Fig.~\ref{fig:convergence_F_Phi}, the spectral radius of $Q^\prime(X)$ is calculated for the values of $\Phi^\prime(\bar y_0)$ and $F^\prime(\bar y_0)$ where the multipliers are real. For better visualization, spectral radius is replaced by $-1$ when the Jacobian has a complex multiplier. The manifolds where the direct and reverse double-multiplier bifurcations occur seem to be affine functions in the $(F^\prime(\bar y_0),\Phi^\prime(\bar y_0))$ plane. As seen before, the multipliers are always real for $\Phi^\prime(\bar y_0)=F^\prime(\bar y_0)=0$. When $F^\prime(\bar y_0)$ decreases from zero to some negative value, lower values of $\Phi^\prime(\bar y_0)$ are required to preserve fixed point stability (cf. Theorem~\ref{th:stability}) and improve convergence without giving rise to oscillating transients. Soon, when $F^\prime(\bar y_0)<-8.3$, the use of amplitude modulation definitely results in complex multipliers. This agrees well with what has been obtained for the case of purely amplitude modulation; cf. Fig.~\ref{fig:hist_opt_K_F}.

%Two of the three real multipliers build a complex conjugate pair when $F^\prime(z_0)$ decreases from zero to some negative value. After that, increasing slopes of $\Phi(z)|_{z=z_0}$ 
\begin{figure}[ht]
\centering 
\includegraphics[width=1.1\linewidth]{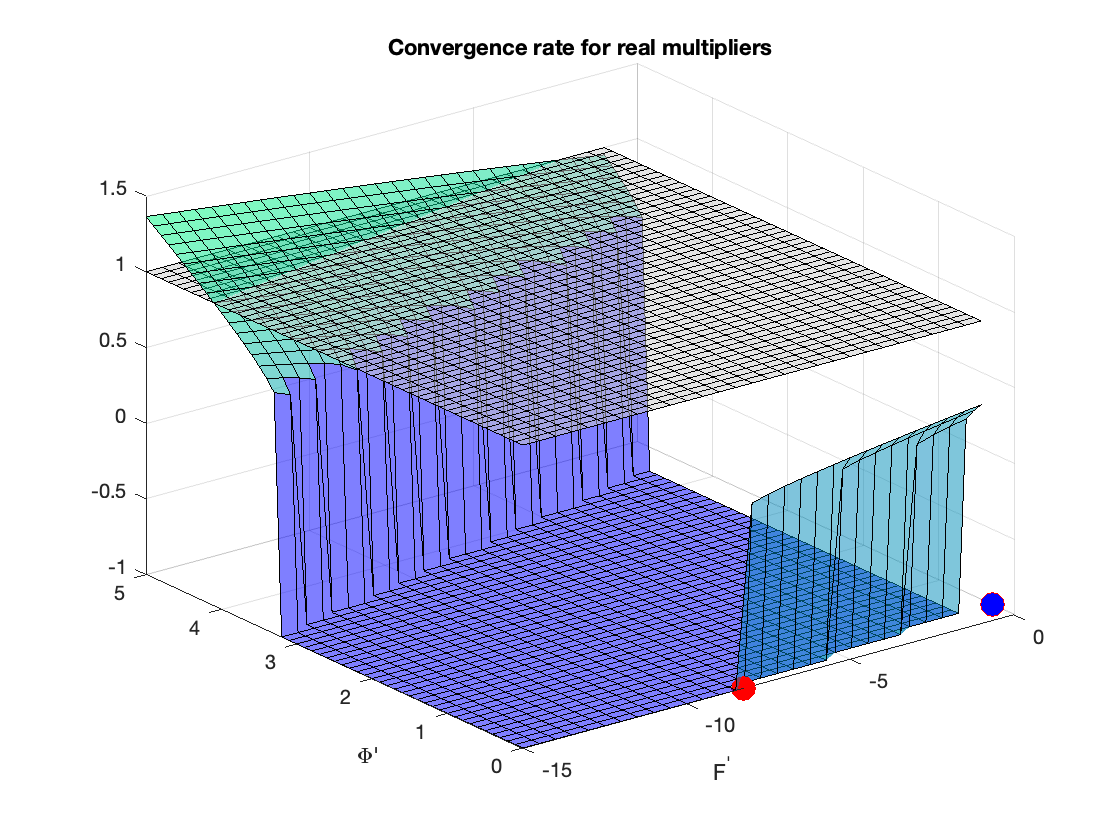}
\caption{%\textcolor{red}{Signed} 
Spectral radius of $Q^\prime(X)$  as function of $\Phi^\prime(\bar y_0)$ and $F^\prime(\bar y_0)$. The green and light-blue surfaces show the spectral radius of the Jacobian when all multipliers are real. For visualization purposes, the spectral radius is set to $-1$ when a complex-conjugate pair of multipliers is present (dark-violet area). The grey plane indicates the stability boundary. The red and blue circles mark, respectively, the optimal pure-amplitude value $F^\prime(\bar y_0)=-8.26$ and the optimal pure-frequency value $\Phi^\prime(\bar y_0)=0.29$, both computed for the population-mean parameter $\alpha=0.0374$.
}\label{fig:convergence_F_Phi}
\end{figure}

\section{Impulsive  dosing controller}\label{sec:case-study}

Using {\it atracurium} in surgical practice, NMB is initiated with a bolus dose of $400\text{--}500~\mathrm{\mu g/kg}$. Further into the procedure, maintenance doses of $80\text{--}200~\mathrm{\mu g/kg}$ are  administered every $10\text{--}20~\mathrm{min}$. Under anesthesia, recovery to $y(t)=25\%$  is achieved $35\text{--}45~\mathrm{min}$ after injection, and recovery is usually $95\%$ complete approximately one hour after injection. This is essentially an open-loop control strategy, but the medication effect is closely observed by means of an NMB monitor and maintained within \eqref{eq:clinical_boundaries}.

%where ${\bf y}_{\min}=2$ and ${\bf y}_{\max}=10$.

\subsection{Open-loop administration}\label{sec:open_loop}

In Fig.~\ref{fig:open_loop}, a scenario with an initial bolus dose of $400~\mathrm{\mu g/kg}$ and a train of five sequential equidistant ($T=20~\mathrm{min}$) maintenance doses of $200~\mathrm{\mu g/kg}$ is simulated for the model~\eqref{eq:lin_NMB},~\eqref{eq:nonlin_NMB} assuming the population mean values of the parameters, i.e. $\bar\alpha$, $\bar\gamma$. Under a sustained maintenance phase ($T=20$, $\lambda=200$), the NMB effect converges to a 1-cycle corresponding to the fixed point 
\begin{equation}\label{eq:fp_NMB}
    X^\top=\lbrack 179.7316 \quad 56.3880 \quad 9.0833\rbrack,
\end{equation}
and yielding an output contained within the corridor $3.9866 \le y(t)\le 6.1562$, which satisfies~\eqref{eq:clinical_boundaries}. The corridor in which the periodic solution evolves can be computed without performing a simulation from Proposition~2 in~\cite{MPZ24}.
The convergence to the 1-cycle is slow, and $\max_{t\in [T,5T]} y(t)=5.6700$, thus revealing that the stationary solution is not reached. There is apparently also an overshoot of $\min y(t)=2.6037$; that is, the system output goes under the lower bound of the corridor corresponding to the 1-cycle. Generally, the performance of the open-loop strategy for the population mean values of the parameters is acceptable since it complies with~\eqref{eq:clinical_boundaries}. Yet, across the considered population, the open-loop strategy fails to deliver the desired performance. The problems are particularly severe with the upper bound of the output corridor; see Fig.~\ref{fig:max_y_ol_pop}, where $\max_{t\in [T,5T]} y(t)\ge {\bf y}_{\max}$ is observed in 18 patient cases, see also Table~\ref{tab:cases}. In four cases (Fig.~\ref{fig:min_y_ol_pop}), $\min_t y(t)>10$ meaning that the output value is never under ${\bf y}_{\max}$. 
\begin{figure}[ht]
\centering 
\includegraphics[width=0.7\linewidth]{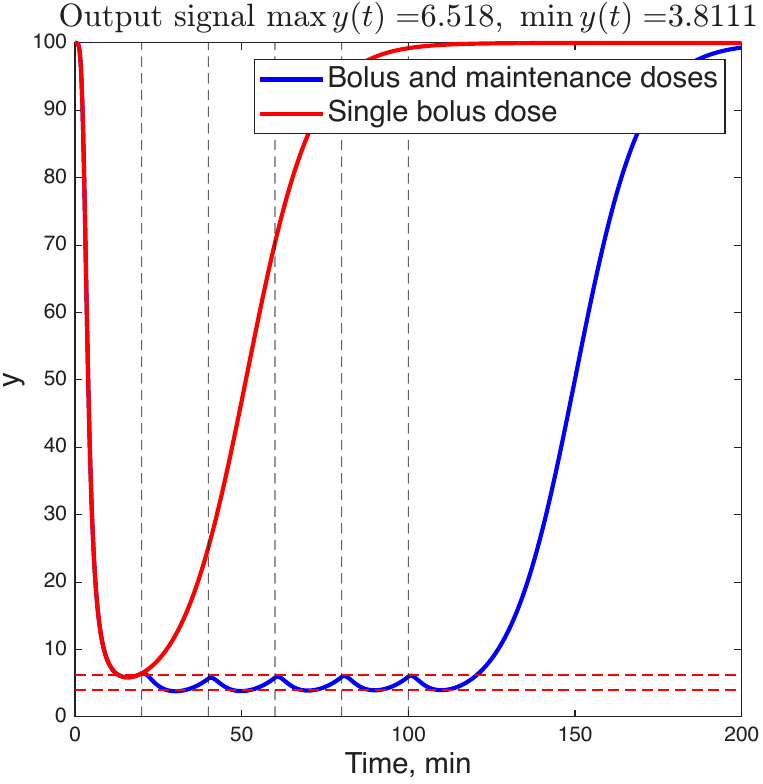}
\caption{Administration of {\it atracurium} in open loop. Population mean values are used in simulation. Response to a bolus dose of $400~\mathrm{\mu g/kg}$ -- solid red line.  A bolus dose of $400~\mathrm{\mu g/kg}$ followed by five maintenance doses each $20~\mathrm{min}$ -- blue line. Vertical dashed lines mark the instants of maintenance dose administration. The output corridor corresponding to a 1-cycle with $T=20~\mathrm{min}$ and $\lambda=200$ is depicted by dashed red lines.}\label{fig:open_loop}
\end{figure}

\begin{figure}[ht]
\centering 
\includegraphics[width=1.0\linewidth]{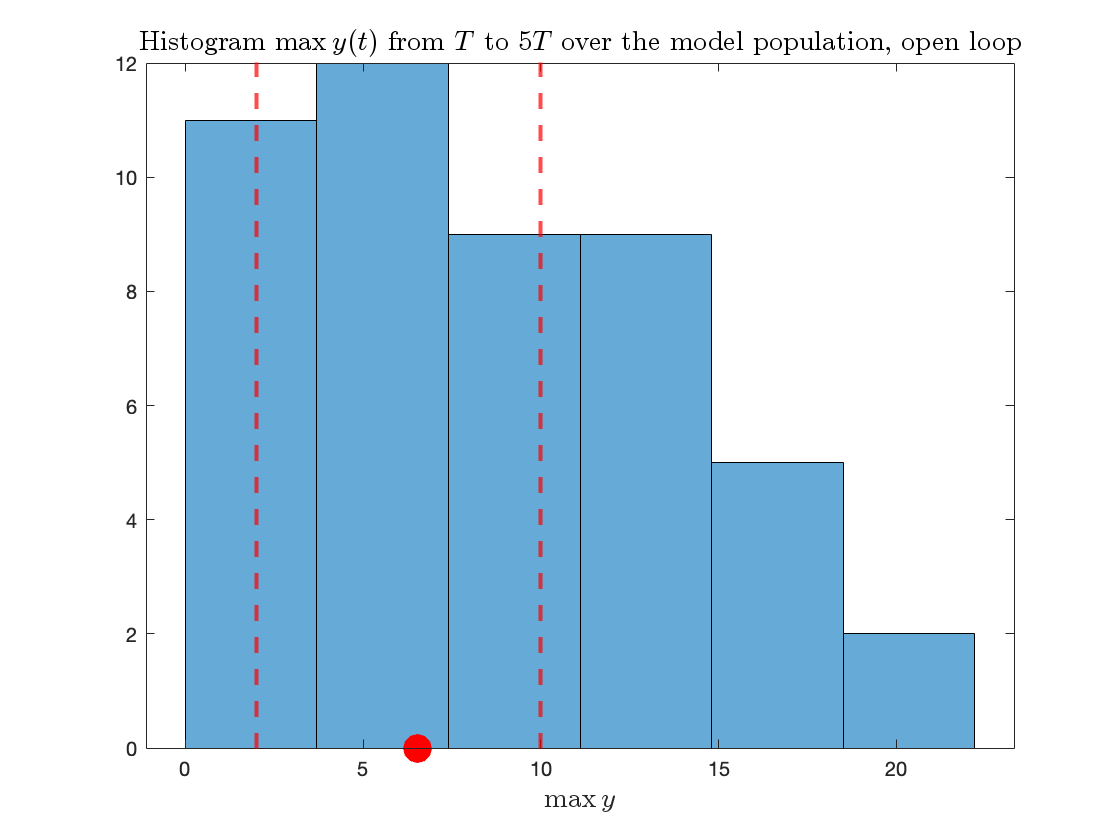}
\caption{Histogram of maximal values of $y(t)$ in the interval $[T,5T]$ over the model population under open-loop administration. A bolus dose of $400~\mathrm{\mu g/kg}$ followed by five maintenance doses $200~\mathrm{\mu g/kg}$ each $20~\mathrm{min}$. Red circle shows the value achieved for population mean values. Desired interval of output values is marked with red vertical dashed lines.}\label{fig:max_y_ol_pop}
\end{figure}
\begin{figure}[ht]
\centering 
\includegraphics[width=1.0\linewidth]{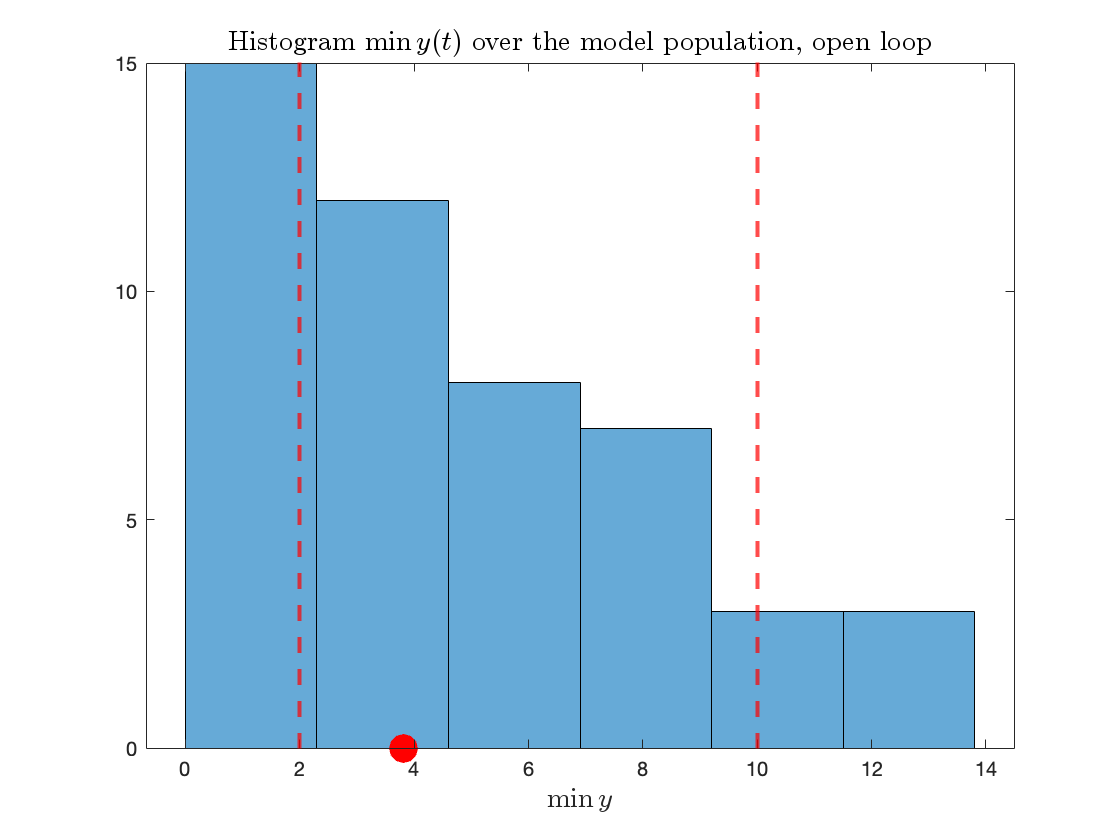}
\caption{Histogram of minimal values of $y(t)$ over the model population under open-loop administration. A bolus dose of $400~\mathrm{\mu g/kg}$ followed by five maintenance doses $200~\mathrm{\mu g/kg}$ each $20~\mathrm{min}$. Red circle shows the value achieved for the population mean values. Desired interval of output values~\eqref{eq:clinical_boundaries} is marked with red vertical dashed lines.}\label{fig:min_y_ol_pop}
\end{figure}

\subsection{Controller design algorithm}\label{sec:design}
In view of the theoretical results presented above, the control problem formulated in Section~\ref{sec:formulation} can be solved in the following way.
\begin{enumerate}[Step 1:]
\item Given the 1-cycle parameters $\lambda$ and $T$, calculate the fixed point $X$ by applying Theorem~\ref{pro:fp}.
\item Select suitable parametrizations of the modulation functions $\bar F(\cdot)$ and $\bar \Phi(\cdot)$.
\item Select a pair $\left(F^\prime(\bar y_0),\Phi^\prime(\bar y_0)\right)$ satisfying the stability condition of 1-cycle from Theorem~\ref{th:stability}.
\item For the selected in Step~2 parametrizations, solve the equations 
\[ \frac{\rm{d}}{\rm{d}\theta} \bar F(\theta)|_{\theta=y_0}=
\frac{F'(\bar y_0)}{\varphi'(\bar y_0)},
%\bar F^\prime(y_0), 
\quad \frac{\rm{d}}{\rm{d}\theta} \bar \Phi(\theta)|_{\theta=y_0}=\frac{\Phi'(\bar y_0)}{\varphi'(\bar y_0)}.
%=\bar \Phi^\prime(y_0)
\]
with respect to the parameters defining the slopes of the modulation functions.
\item Solve the equations 
\[
F(\bar y_0)=\lambda, \quad \Phi(\bar y_0)=T,
\]
with respect to the parameters of $\bar F(\cdot),\bar \Phi(\cdot)$.
\end{enumerate}
Recall that the choice of the slopes of the modulation functions in Step~3 influences, besides stability of the periodic solution, the transient properties of the closed-loop system when converging to the 1-cycle from an initial condition outside of it.
\subsection{Implementation}\label{sec:implementation}
The pulse-modulated controller design procedure specified above is applied in this section to the NMB model in~\eqref{eq:lin_NMB},~\eqref{eq:nonlin_NMB} assuming the population mean values of the parameters.
%Following~\cite{MPZ24}, let the modulation functions be selected as
%\[
%F(\xi)\triangleq (\bar F \circ \varphi)(\xi), \quad \Phi(\xi)\triangleq (\bar \Phi \circ \varphi)(\xi),
%\]
%where $\bar F (\cdot), \bar \Phi (\cdot)$ are the functions to be selected in order to enforce the desired orbitally stable 1-cycle in the closed-loop system. 
\paragraph*{Step~1}
The fixed point corresponding to the 1-cycle parameters $\lambda=200$, $T=20$ is given by~\eqref{eq:fp_NMB}.
\paragraph*{Step~2}
%A parametrization of the modulation functions is required in Step~2 of the controller design algorithm.
Following~\cite{MPZ24}, select $\bar F (\cdot), \bar \Phi (\cdot)$  in~\eqref{eq.phi_f_nobar} as piecewise affine, i.e.
\begin{align}\label{eq:affine_Phi}
    \bar \Phi (\xi)= \begin{cases} \Phi_2 &   \Phi_2 < k_2\xi +k_1, \\
     k_2\xi +k_1 & \Phi_1 \le  k_2\xi +k_1 \le \Phi_2, \\
    \Phi_1  &  k_2\xi +k_1 < \Phi_1, 
     \end{cases}
\end{align}
\begin{align}\label{eq:affine_F}
    \bar F (\xi)= \begin{cases} F_1 &  k_4\xi +k_3< F_1, \\
     k_4\xi +k_3 & F_1 \le k_4\xi +k_3 \le F_2, \\
    F_2 & F_2 <k_4\xi +k_3.
     \end{cases}
\end{align}
Recalling from~\eqref{eq:nonlin_NMB} that $y(t)\in\lbrack 0, 100\rbrack$, the following inequalities apply
\begin{equation}\label{eq:F_Phi_bounds}
    \Phi_1\le k_1, \  100k_2+k_1\le\Phi_2, \  F_1\le k_3, \  100k_4+k_3\le F_2.
\end{equation}
From the bounds on the modulation functions, it follows that the feedback cannot administer a dose that is greater than $F_2$ or less than $F_1$. Further, no dose is administered sooner than $\Phi_1$ from the previous one and at least one dose is administered within a time interval of $\Phi_2$. Numerical values of these bounds can be easily obtained from the manual medication protocols for the drug in question.

The affine form of the modulation functions results in linear design equations but has an apparent limitation. The functions $\bar F(\cdot)$ and $\bar \Phi(\cdot)$ are completely defined by the fixed point $X$ (i.e. the values of $\lambda$, $T$) and the  slopes $k_4$, $k_2$ calculated in {\it Step~4} to guarantee stability. If an additional dosing condition has to be satisfied, e.g. a certain bolus dose has to be administered in the beginning of the surgery ($\bar F(100\%)$), then a compromise has to be made or a more accommodating parametrization of the  modulation functions selected.

\paragraph*{Step~3}
Choose the slopes of the modulation functions at the fixed point as $F^\prime(\bar y_0)=-2$ and $\Phi^\prime(\bar y_0)=0.7$. 
For stability condition~\eqref{eq.necess-stab1}, it applies
\[
0.8156 > -1.2825\cdot 10^{9}.
\]
Stability condition~\eqref{eq.necess-stab2} then reads
\[
 0.0454\cdot F^\prime(\bar y_0) -0.5700\cdot\Phi^\prime(\bar y_0)=-0.4898>-1.
\]
Condition~\eqref{eq.necess-stab3} gives
\[
 5.1721\cdot 10^{-10} >0
\]
and orbital stability  of the periodic solution is therefore guaranteed according to Theorem~\ref{th:stability}. The spectral radius of the Jacobian $\rho(Q^{\prime}(X))= 0.2349$ and the eigenvalue spectrum $\sigma(Q^{\prime}(X))= \{ 0.1551 \pm 0.1765i,\  0.0002\}$.
\paragraph*{Step~4}
For $F^\prime(\cdot)$ and $\Phi^\prime(\cdot)$ in~\eqref{eq:affine_Phi} and~\eqref{eq:affine_F}, by applying the chain rule 
\begin{align}\label{eq:k_2_k_4}
     F^\prime(\bar y_0)&=\bar F^\prime(y_0)\varphi^\prime(\bar y_0)= k_4 \varphi^\prime(\bar y_0),\\ 
     \Phi^\prime(\bar y_0)&= \bar \Phi^\prime(y_0)\varphi^\prime(\bar y_0)= k_2 \varphi^\prime(\bar y_0), \nonumber
\end{align}
where $\bar y_0=CX$ and
\[
\varphi^\prime(\xi)= -\frac{\gamma 100 C_{50}^\gamma  \xi^{\gamma-1}}{ {(C_{50}^\gamma + {\xi}^\gamma)}^2}.
\]
This yields the numerical values $\bar y_0=9.0833$, $\varphi^\prime(\bar y_0)=-1.6616$,
$k_4 =1.2036$, $k_2 =-0.4213$.
\paragraph*{Step~5}
To obtain a 1-cycle with the desired parameters, the following equations have to hold
\begin{align}\label{eq:k_1_k3}
    F(\bar y_{0})&= (\bar F \circ \varphi)(\bar y_{0})=k_4 \varphi(\bar y_{0})+k_3= \lambda,\\
    \Phi(\bar y_{0})&= (\bar \Phi \circ \varphi)(\bar y_{0})= k_2 \varphi(\bar y_{0})+k_1=T, \nonumber
\end{align}
thus yielding $k_3=192.7539$, $k_1=22.5361$. 

Now the pulse-modulated feedback controller is designed and the resulting modulation functions are depicted in Fig.~\ref{fig:bar_F_bar_Phi}.

\subsection{Feedback controller evaluation}

Consider now the feedback implementation of the periodic dosing regimen from Section~\ref{sec:open_loop}. The target 1-cycle is the one defined by the fixed point in~\eqref{eq:fp_NMB}, corresponding to $\lambda=200$ and $T=20$. Since the modulation functions are designed to satisfy~\eqref{eq:k_1_k3}, this periodic orbit coincides with that generated by the corresponding open-loop train of equidistant impulses of equal weight. The resulting periodic solution of the closed-loop system~\eqref{eq:lin_NMB},~\eqref{eq:nonlin_NMB},~\eqref{eq:2} is shown in Fig.~\ref{fig:1-cycle}.

%Consider the same 1-cycle  as obtained in the open-loop administration (Section~\ref{sec:open_loop}), i.e. the one defined by the fixed point in~\eqref{eq:fp_NMB}, but implemented in closed-loop system~\eqref{eq:lin_NMB},~\eqref{eq:nonlin_NMB},~\eqref{eq:2}.
%A periodic solution of the closed-loop system corresponding to this fixed point  is shown in Fig.~\ref{fig:1-cycle}. 

 A feedback drug administration system is expected to work in a wide range of the plant output values. Transient behaviors to the designed periodic solution from a distant point in the state space are therefore important. An underdosing of an NMB agent, i.e. $\exists~ t: {\bf y}_{\max}< y(t)$, is more critical than an overdosing event, i.e. $\exists~ t:  y(t)<{\bf y}_{\min}$, since there is a risk of disrupting the surgery in the former case.
 For the design case considered in Section~\ref{sec:implementation}, the transient process corresponding to a start of a surgical procedure, i.e. $x(0)=0,\ y(0)=100\%$, is shown in Fig.~\ref{fig:transient_7}. The controller satisfies the clinical output range in~\eqref{eq:clinical_boundaries} but exhibits an overshoot with respect to the bounds of the 1-cycle resulting from the complex eigenvalues of the Jacobian, cf. Step~3 in Section~\ref{sec:implementation}.  This can be addressed by reducing the slope of the frequency modulation function to the value corresponding to the double multiplier point and maximizing the (local) convergence rate, cf. Appendix~\ref{app:E}. Then, the overshoot is reduced from $\inf_t y(t)=2.9689$ to $\inf_t y(t)=3.2252$, i.e. from $25.5\%$ to $19\%$ with respect to lower output corridor bound. This is in fact worse  compared to open-loop administration in Fig.~\ref{fig:open_loop}, where the overshoot is only $4.4\%$. Yet, in open loop, there is an undershoot that is not at all present in closed-loop administration. Notice that the performance of the pulse-modulated controller can be further tuned for a particular combination of $(\alpha,\gamma)$ but is not performed here as the clinical output range in~\eqref{eq:clinical_boundaries} is well satisfied.

\textcolor{black}{As mentioned in Section~\ref{sec:data}, some of the models in the considered dataset were invalidated in \cite{MP26} as they produced infeasible maintenance phase regimens. Compared to what is observed in clinical practice, the subset of PIN entries $\Gamma_1=\{4,     6,    11,    12,    15,    22,    24,    25,    37,    46\}$ resulted in higher NMB agent  doses   as well as longer interdose intervals. Elevated maintenance doses with normal interdose intervals were obtained for the models $\Gamma_2=\{ 2,     9,    26,    39,    43,    44,    45\}$.}

To study the performance of the pulse-modulated NMB agent  controller across the  patient model population, four design cases leading to different modulation functions are now considered, see Table~\ref{tab:cases}.

Open-loop administration is denoted as Case~0. The evaluation criteria are the incidence of NMB agent underdosing and overdosing defined according to~\eqref{eq:clinical_boundaries}.
Histograms of the minimal and maximal values of the PK/PD plant output for the different design cases are provided in Fig.~\ref{fig:min_y} and Fig.~\ref{fig:max_y}, correspondingly. Notice that the maximal value of $y(t)$ is calculated over the interval $t\in \lbrack T,5T \rbrack$ since it always holds that $\sup_t y(t)=y(0)=100\%$. \textcolor{black}{Importantly, all but one (PIN=10) cases of underdosing in open-loop administration occur for invalidated models that belong either to $\Gamma_1$ or $\Gamma_2$. This confirms that the standard manual {\it atracurium} administration procedure is generally safe with respect to underdosing. }%\marginpar{\textcolor{blue}{Q20.3\\QAE}}

By inspection of the histograms, it can be concluded that the introduction of the feedback has a major impact on the underdosing events but not on the overdosing ones. Actually, pulse-modulated feedback either performs similarly to the open-loop administration (or even worse, \textcolor{blue}{cf. Case~1}) since the first dose is always provided at $t=0$ and is equal to the initial bolus dose according to the evaluation protocol. A straightforward way of handling this issue is to start with lower bolus doses and  allow the feedback to activate earlier (i.e. use a lower $\Phi_1$). This will of course prolong the induction phase of NMB.
Applying the open-loop administration (Case~0), underdosing is observed at 18 instances across the cohort (Table~\ref{tab:cases}), which is reduced to 11 instances under feedback (Case~2). In Case~1, the highest values of the output are clustered about the population mean value. It is also the only design case where the lowest values of the output are never over the bound ${\bf y}_{\max}$. In contrast, the open-loop administration yields \textcolor{black}{four cases} where the output $y(t)$ lies over the bound ${\bf y}_{\max}$ all the time, i.e. $\forall t: {\bf y}_{\max}< y(t)$. \textcolor{black}{All the cases where underdosing is observed under feedback belong to either $\Gamma_1$ or $\Gamma_2$, i.e. correspond to invalidated models.}%\marginpar{\textcolor{blue}{Q20.3\\QAE}}

The robustness properties of the proposed pulse-modulated controller are illustrated in Fig.~\ref{fig:alpha_gamma_color}. As expected, controlling models with very high or very low values of $\gamma$ with the same controller presents significant challenges. Low values of $\gamma$ lead to underdosing and high values of $\gamma$ result in overdosing. Combinations of low $\alpha$ and extreme values of $\gamma$ are especially difficult to deal with and always exhibit over- or underdosing under an impulsive feedback designed for the population mean values. Recall that the open-loop strategy fails to maintain clinical NMB interval~\eqref{eq:clinical_boundaries} in two thirds of the modeled cases (i.e. 32 out of 48, see Tab.~\ref{tab:cases}). \textcolor{black}{The introduction of impulsive feedback reduces the occurrence of underdosing both for models in $\Gamma_1$ and those in $\Gamma_2$, even though these models are judged infeasible. }%\marginpar{\textcolor{blue}{Q20.3}}

\begin{figure}[ht]
\centering 
\includegraphics[width=1.0\linewidth]{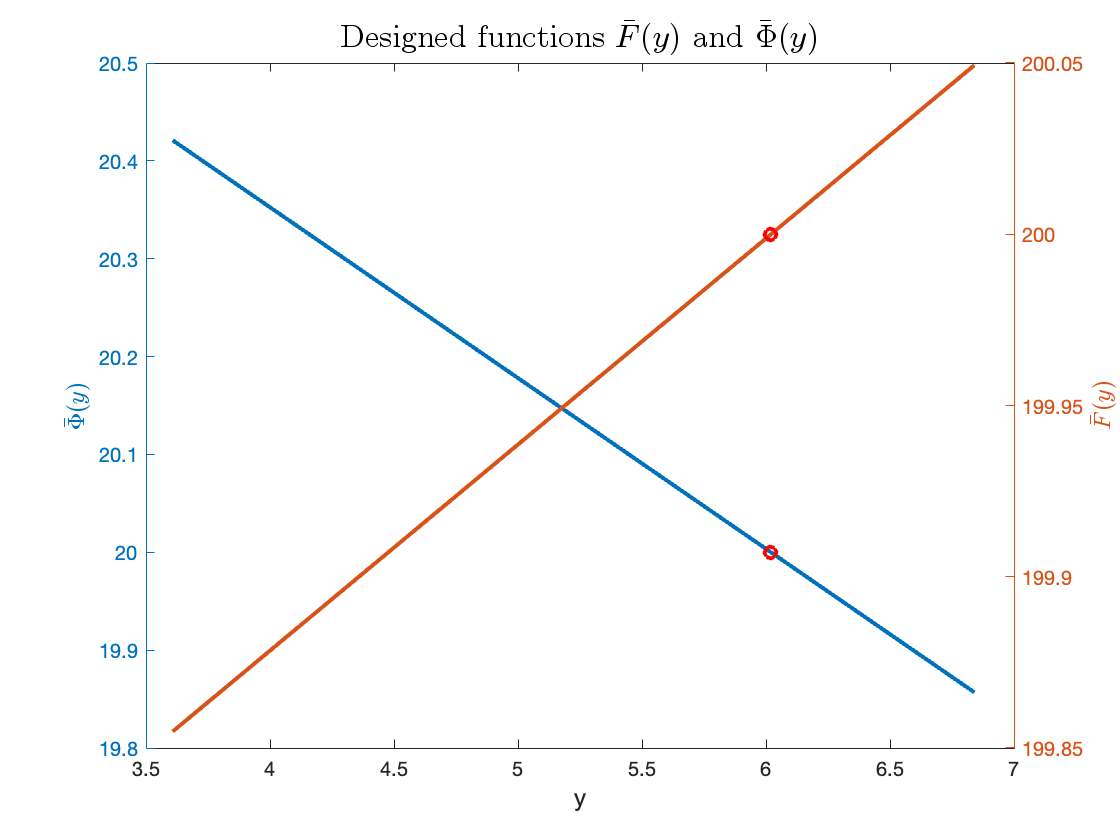}
\caption{The affine parts of the designed functions $\bar F(\cdot)$, $\bar \Phi(\cdot)$. The values $\bar F(y_0)$, $\bar \Phi(y_0)$ are marked with red circle.}\label{fig:bar_F_bar_Phi}
\end{figure}
\begin{figure}[ht]
\centering 
\includegraphics[width=1.0\linewidth]{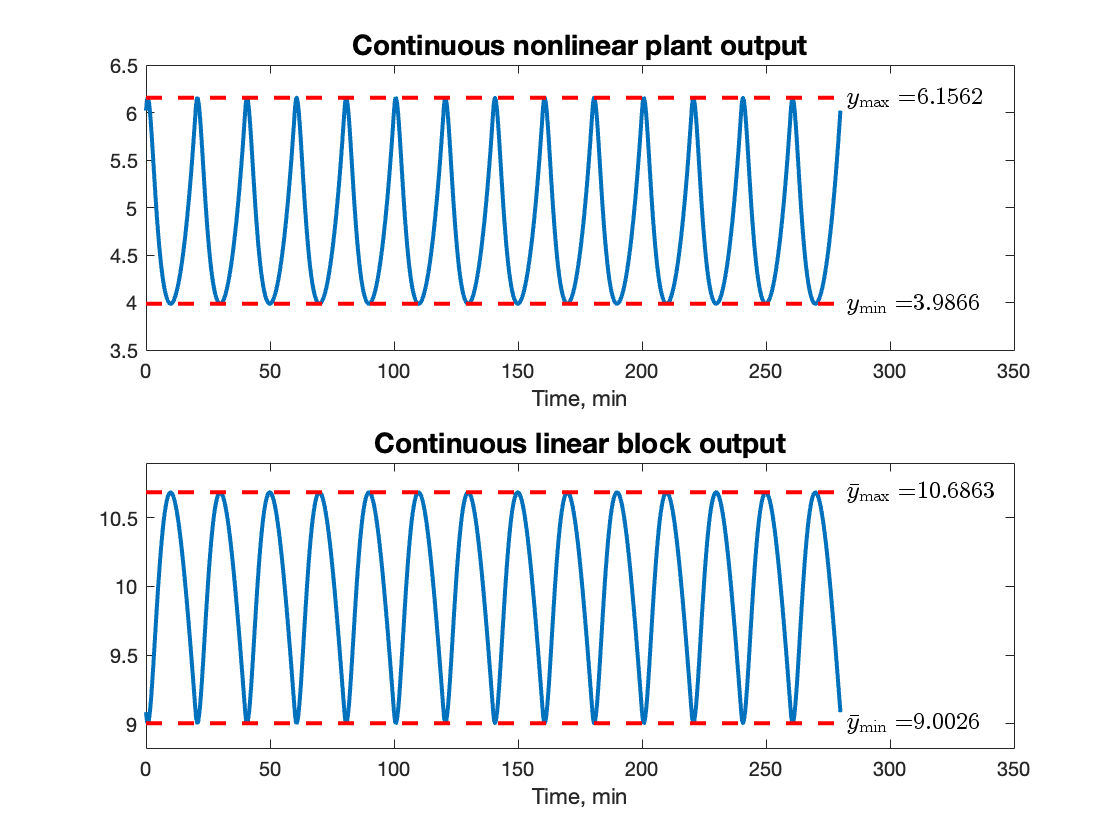}
\caption{The designed 1-cycle initiated from the fixed point $X$. The nonlinear output $y(t)$ (top plot)  and the linear output $\bar y(t)$ (bottom plot) are presented. }\label{fig:1-cycle}
\end{figure}
\begin{figure}[ht]
\centering 
\includegraphics[width=1.0\linewidth]{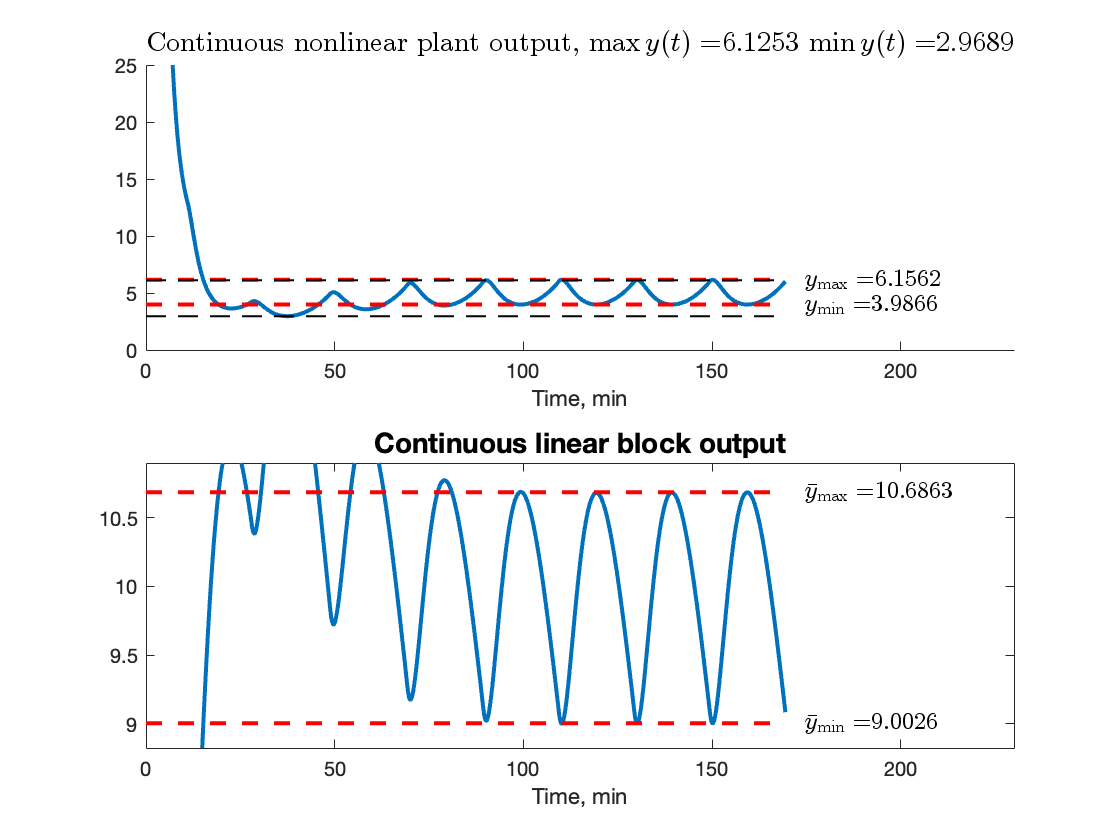}
\caption{The transient process from $x(0)=0, \ y(0)=100\%$ to the designed 1-cycle for $F^\prime(\bar y_0)=-2$, $\Phi^\prime(\bar y_0)=0.7$, Case~2 in Table~\ref{tab:cases}. The nonlinear output $y(t)$ (top plot)  and the linear output $\bar y(t)$ (bottom plot) are presented. The horizontal red dashed lines show the output corridor bounds for the nonlinear/linear output. The horizontal black dashed lines mark $\inf_t y(t)$ and $\sup_{t\in \lbrack T,5T\rbrack} y(t)$.}\label{fig:transient_7}
\end{figure}
\begin{figure}[ht]
\centering 
\includegraphics[width=1.0\linewidth]{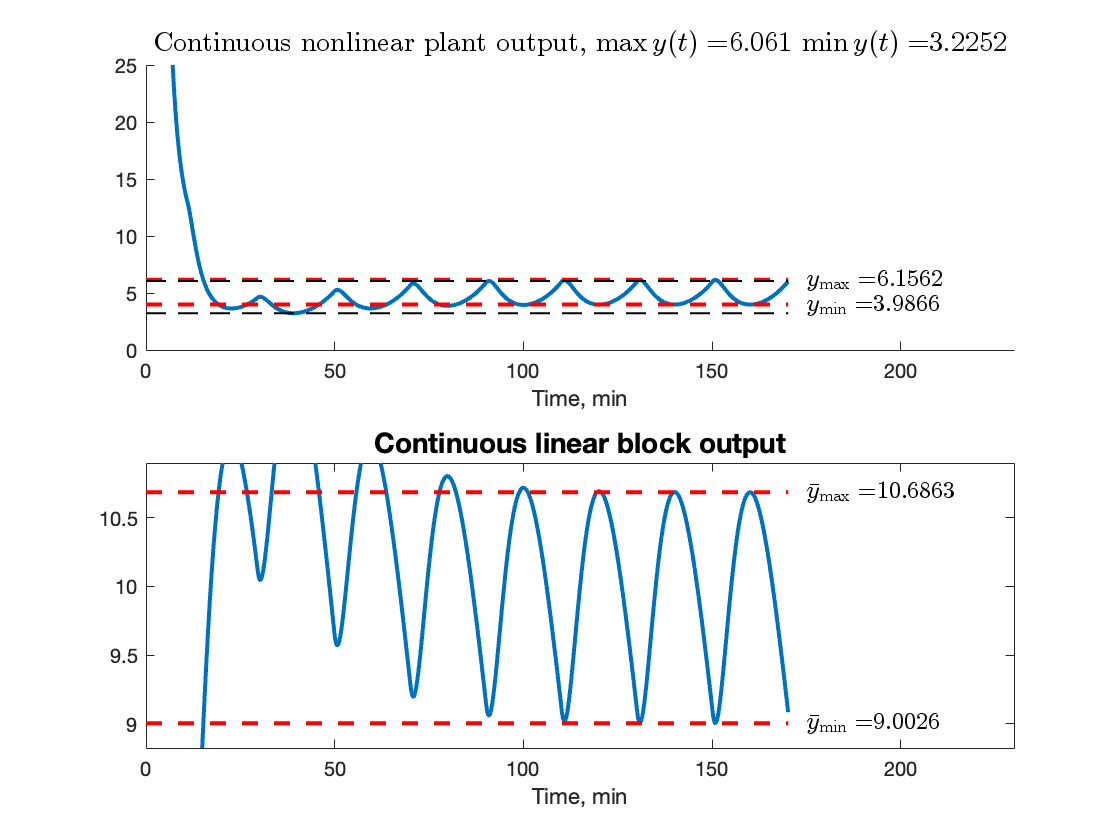}
\caption{The transient process from $x(0)=0, \ y(0)=100\%$ to the designed 1-cycle for $F^\prime(\bar y_0)=-2$, $\Phi^\prime(\bar y_0)=0.33707$ (double multiplier point), Case~3 in Table~\ref{tab:cases}. The nonlinear output $y(t)$ (top plot)  and the linear output $\bar y(t)$ (bottom plot) are presented. The horizontal red dashed lines show the output corridor bounds for the nonlinear/linear output. The horizontal black dash lines mark $\inf_t y(t)$ and $\sup_{t\in \lbrack T,5T\rbrack} y(t)$.}\label{fig:transient_8}
\end{figure}
\begin{figure}[h]
\subcaptionbox{Case 1}{\includegraphics[width=0.24\textwidth]{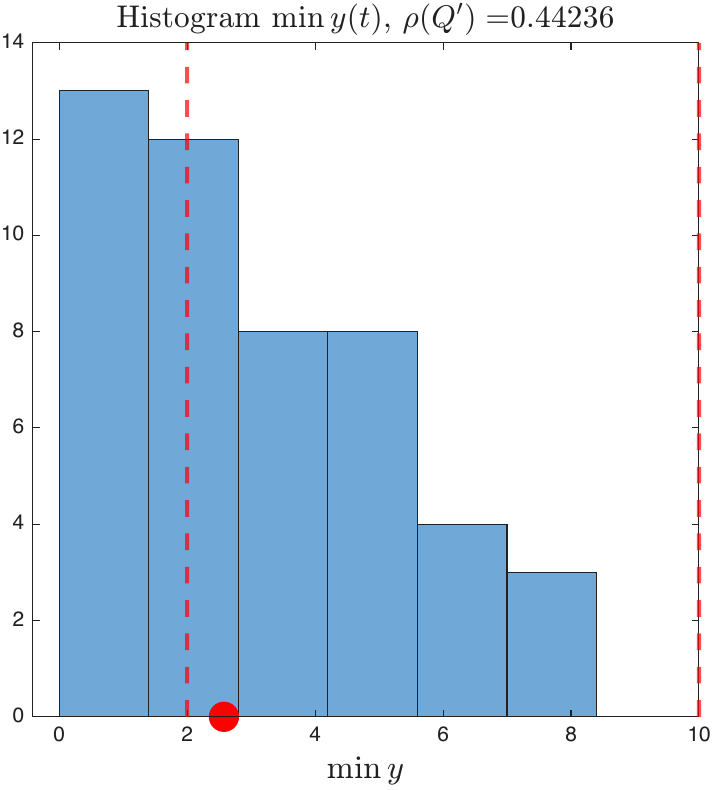}}%
\hfill
\subcaptionbox{Case 2}{\includegraphics[width=0.24\textwidth]{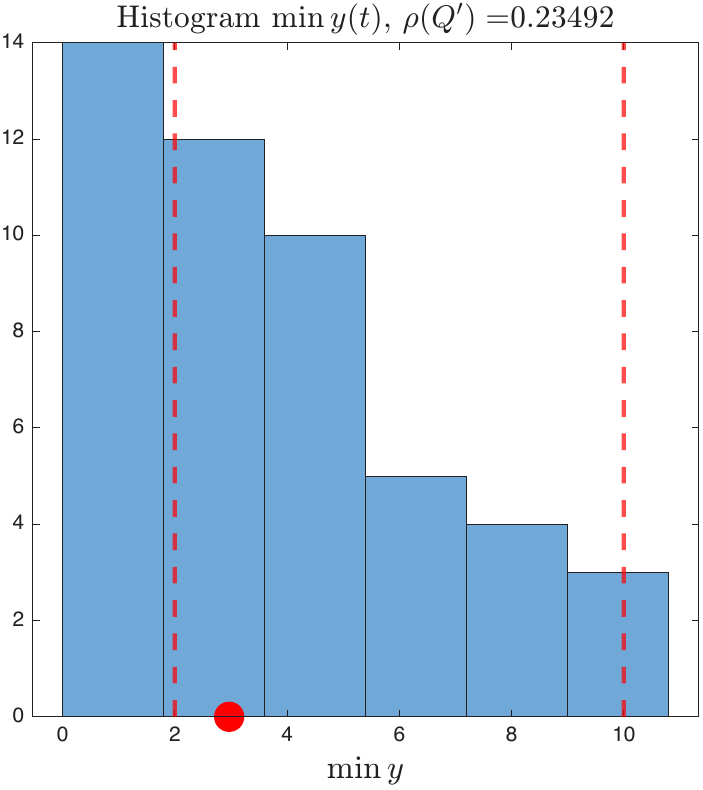}}%
\\
\subcaptionbox{Case 3}{\includegraphics[width=0.24\textwidth]{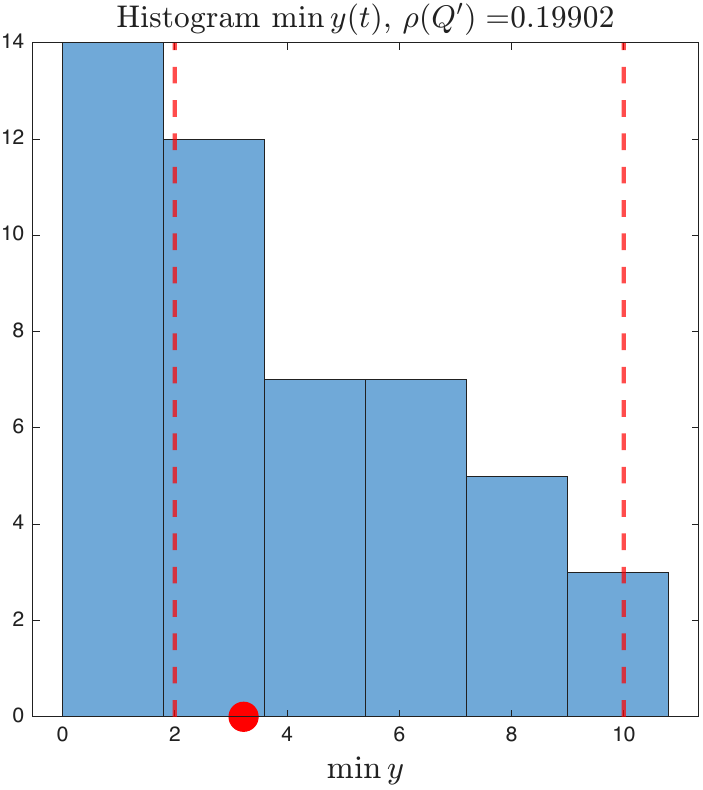}}%
\hfill
\subcaptionbox{Case 4}{\includegraphics[width=0.24\textwidth]{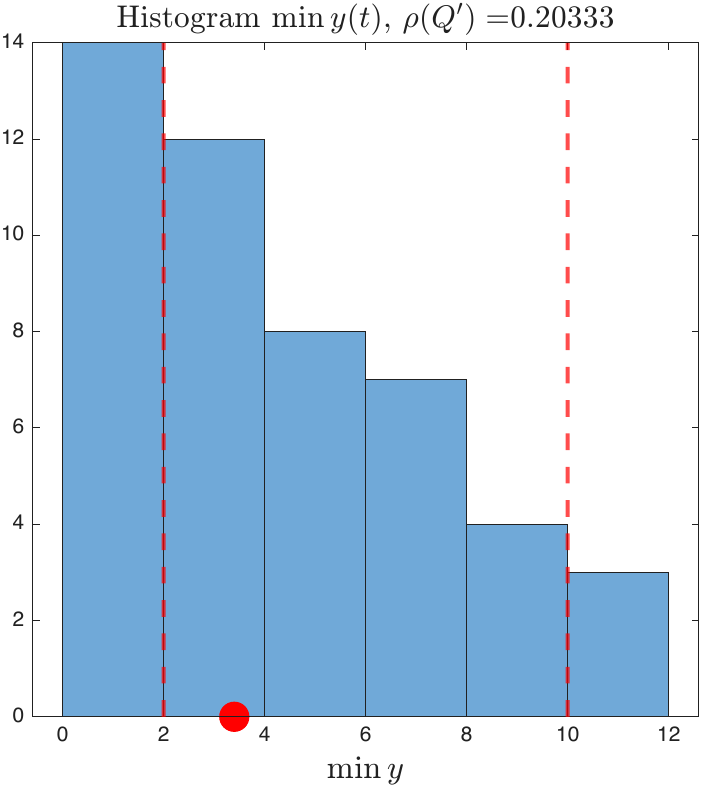}}%
\caption{Histograms of $\inf_t y(t)$ across the patient cohort for the design cases in Table~\ref{tab:cases}. The vertical dashed lines correspond to ${\bf y}_{\min}$, ${\bf y}_{\max}$ in~\eqref{eq:clinical_boundaries}.}\label{fig:min_y}
\end{figure}
\begin{figure}[h]
\subcaptionbox{Case 1}{\includegraphics[width=0.24\textwidth]{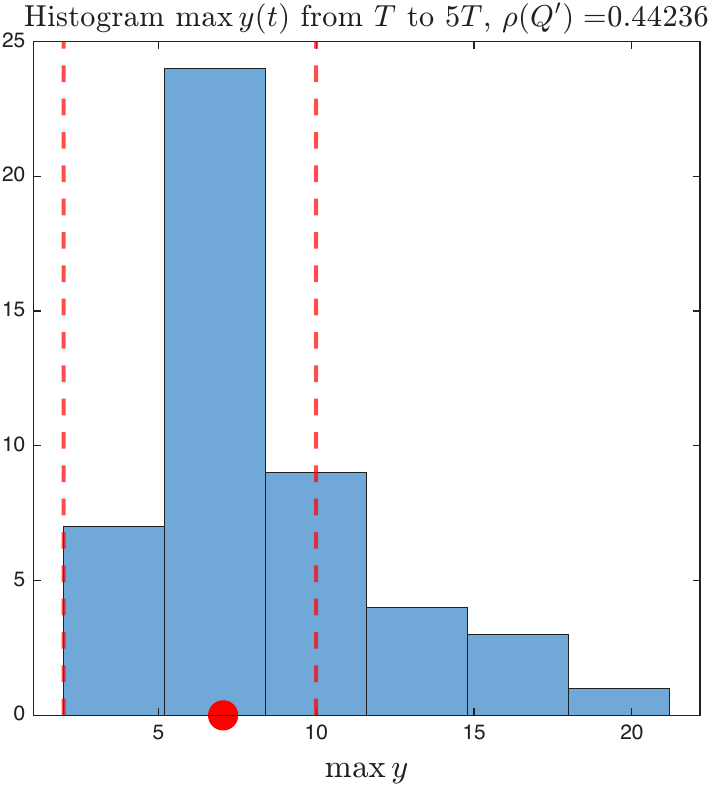}}%
\hfill
\subcaptionbox{Case 2}{\includegraphics[width=0.24\textwidth]{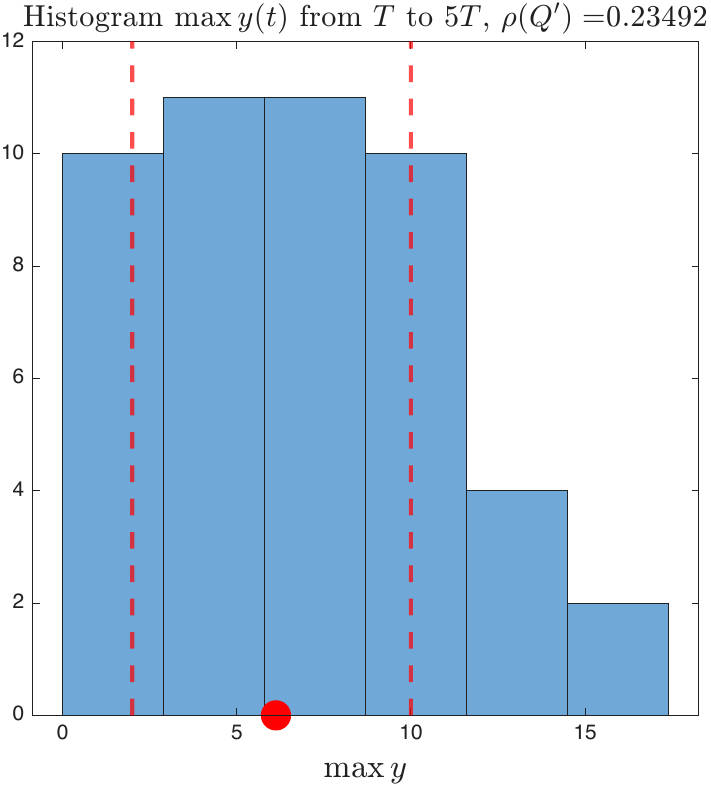}}%
\\
\subcaptionbox{Case 3}{\includegraphics[width=0.24\textwidth]{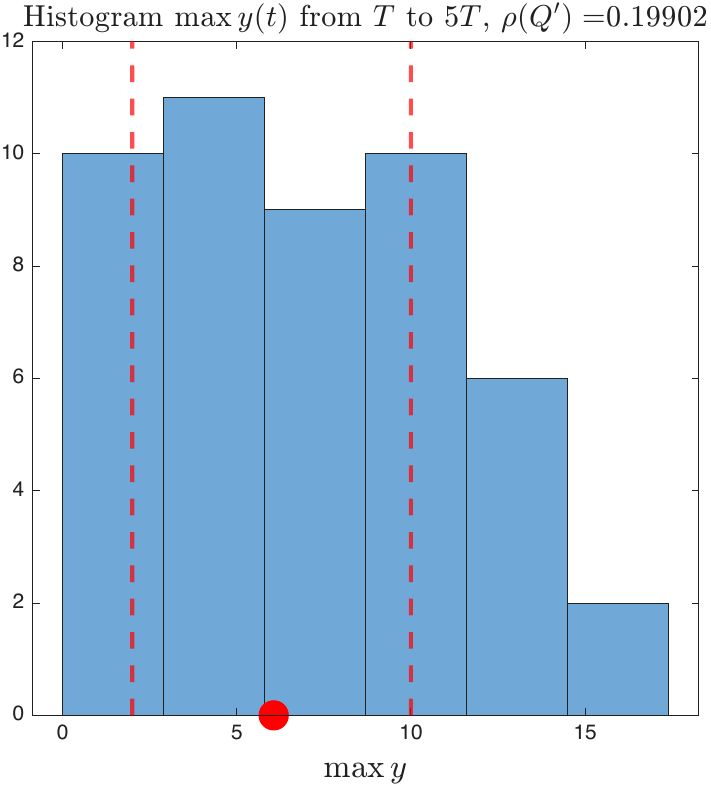}}%
\hfill
\subcaptionbox{Case 4}{\includegraphics[width=0.24\textwidth]{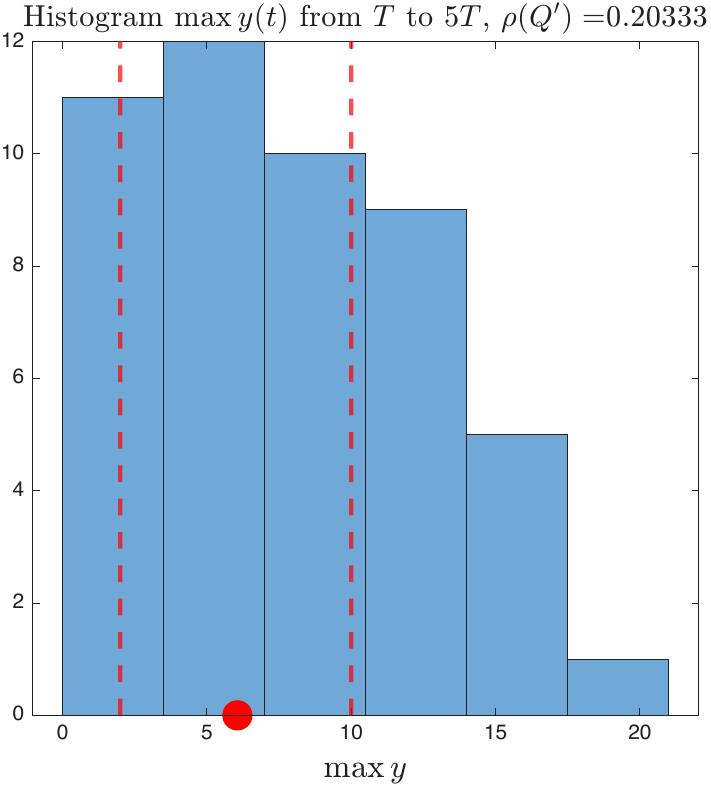}}%
\caption{Histograms of $\sup_{t\in \lbrack T,5T\rbrack} y(t)$ across the patient cohort for the design cases in Table~\ref{tab:cases}. The vertical dashed lines correspond to ${\bf y}_{\min}$, ${\bf y}_{\max}$ in~\eqref{eq:clinical_boundaries}.}\label{fig:max_y}
\end{figure}
%

%\begin{figure}[ht]
%\centering 
%\includegraphics[width=0.7\linewidth]{fig/case_1-3.png}
%\caption{Transient process to the 1-cycle for three combinations of modulation functions slopes. Case~1: $F^\prime(\bar y_0)=-0.1, %\Phi^\prime(\bar y_0)=0.29$; 
%Case~2: $F^\prime(\bar y_0)=-0.1, \Phi^\prime(\bar y_0)=0.35$; Case~3: $F^\prime(\bar y_0)=-0.1, \Phi^\prime(\bar y_0)=0.4$.}\label{fig:case_1-3}
%\end{figure}
%{\appendix[Proof of the Zonklar Equations]
%TODO, if needed}
\begin{table}[]
\begin{center}
\begin{tabular}{ |c|c|c|l|l| } 
 \hline
  & $F^\prime; \Phi^\prime$ & $\rho(Q^\prime)$ & PIN, $y<\bf y_{\min}$ & PIN, $y>\bf y_{\max}$\\ 
  \hline\hline
  Case 0& 0;\,0 &  0.4829  & 1,7,16,17,21,27,  &{\sl 2},{\bf 4},{\bf 6},{\sl 9},10,{\bf 11},{\bf 12}, \\ 
         &       &         & 28,29,32,33,36,   &{\bf 15},{\bf 22},{\bf 24},{\bf 25},{\sl 26},{\bf 37},\\ 
         &       &         & 40,41,48                         &{\sl 39},{\sl 43},{\sl 44},{\sl 45},{\bf 46} \\ 
 Case 1& -1;\,4 & 0.4424 &1,7,13,14,16,17, & {\bf 4},{\bf 6},{\bf 11},{\bf 12},{\bf 15},{\bf 22},{\bf 24},\\ 
        &        &       &19,21,27,28,29,   & {\bf 25},{\sl 26},{\bf 37},{\sl 43},{\bf 46}  \\ 
        &       &      &32,33,36,40,41,48   &  \\ 
 Case 2& -2;\,0.7 & 0.2349  &1,7,16,17,21,27,   &{\sl 2},{\bf 6},{\bf 11},{\bf 24},{\bf 25},{\sl 26},{\bf 37},\\ 
        &          &        & 28,29,32,33,36,40, & {\sl 39},{\sl 43},{\sl 44},{\bf 46}  \\
        &          &       &41,48                  &\\
 Case 3& -2;\,0.337 & 0.1990  &1,7,16,17,21,27,    &{\sl 2},{\bf 6},{\sl 9},{\bf 11},{\bf 12},{\bf 24},{\bf 25}, \\ 
        &              &        &28,29,32,33,36, &{\sl 26},{\bf 37},{\sl 39},{\sl 43},{\sl 44},{\bf 46} \\ 
        &              &        & 40,41,48               & \\
 Case 4& -1;\,0.392 & 0.2033 &1,7,16,17,21,27,  &{\sl 2},{\bf 4},{\bf 6},{\sl 9},{\bf 11},{\bf 12},{\bf 15},\\
        &             &       &28,29,32,33,36, & {\bf 22},{\bf 24},{\bf 25},{\sl 26},{\bf 37},{\sl 39},\\
        &             &       & 40,41,48                    &{\sl 43},{\sl 44},{\bf 46}\\
 \hline
\end{tabular}
\end{center}
 \caption{Pulse-modulated controller design cases. \textcolor{black}{PINs of invalidated models in $\Gamma_1$ are typeset in boldface, while PINs belonging to $\Gamma_2$ are typeset in slanted.}}
    \label{tab:cases}
\end{table}
% High lambda, high T: 4     6    11    12    15    22    24    25    37    46
% High lambda, normal T: 2     9    26    39    43    44    45

\begin{figure}[ht]
\centering 
\includegraphics[width=1.0\linewidth]{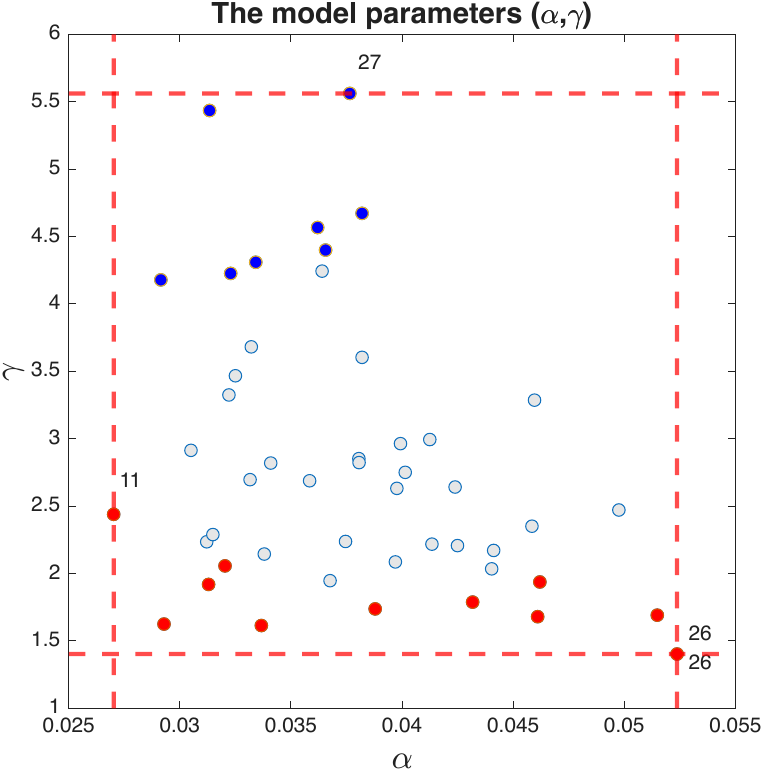}
\caption{The model parameter pairs in the data set. The controller design in Case~2 is applied. Underdosing (${\bf y}_{\max}<\sup_{t\in [T,5T]}y(t)$) -- red dots. Overdosing ($\inf_{t}y(t)<{\bf y}_{\min}$) -- blue dots.  The extreme parameter values are indicated by the Patient Identification Number. 
%\textcolor{red}{In the text, under/overdosing meant the opposite: $y<y_{max}$ and $y<y_{min}$, respectively.}
}\label{fig:alpha_gamma_color}
\end{figure}

\section{Conclusions}\label{sec:conclude}
A pulse-modulated feedback controller employing both amplitude and frequency modulation is considered. The closed-loop dynamics of the controller governing a continuous time-invariant single-input single-output nonlinear plant of Wiener structure are captured by a discrete map. A controller design procedure to produce a sustained stationary periodic solution with given parameters is proposed. It is based on calculating the fixed point of the discrete map corresponding to the desired periodic solution and stabilizing it by choosing the slopes of the modulation functions. A necessary and sufficient stability condition for a periodic solution with only one firing of the feedback in the least period is proven. The utility of the proposed controller is illustrated by a dosing application of a neuromuscular blockade agent widely applied in anesthesiology. A simulation study on a database of pharmacokinetic-pharmacodynamic models estimated from clinical data shows a significant improvement in the incidence of underdosing events compared to an open-loop administration scheme. To reduce the incidence of overdosing events in the induction phase, a high bolus dose has to be substituted by a sequence of lower doses, which would prolong the NMB induction phase.

\bibliographystyle{IEEEtran}
\bibliography{observer,refs} 

\appendices

\section{The Opitz formula for $3\times 3$ matrices}\label{app.opitz}

Given a function $f$, which is complex-analytic in a vicinity of the spectrum $\{-a_1,-a_2,-a_3\}$ of the matrix $A$ from~\eqref{eq:1}, the Opitz formula~\cite{DeBoor2005,PRM24} states that
\begin{equation*}%\label{eq:Opitz}
f(A)=
\begin{bmatrix}
f(-a_1) & 0 & 0\\
g_1f[-a_1,-a_2] & f(-a_2) & 0\\
g_1g_2f[-a_1,-a_2,-a_3] & g_2f[-a_2,-a_3] & f(-a_3)
\end{bmatrix}.
\end{equation*}
%Here, the \emph{divided differences} are found as follows
%\begin{align*}
%f[x_0,x_1]&=\frac{f(x_1)-f(x_0)}{x_1-x_0},\notag\\
%f[x_0,x_1,x_2]&=\frac{f(x_0)}{(x_2-x_0)(x_1-x_0)}+\frac{f(x_1)}{(x_2-x_1)(x_0-x_1)}+\notag\\
%&+\frac{f(x_2)}{(x_0-x_2)(x_1-x_2)}.\notag
%\end{align*}
%Here $x_0,x_1,x_2$ are pairwise distinct. Notice that the divided differences are symmetric in their arguments, e.g., $f[x_0,x_1,x_2]=f[x_1,x_2,x_0]$.

Whereas the Opitz formula is valid for two-diagonal matrices of arbitrary dimensions~\cite{DeBoor2005}, the three-dimensional case can be proved in a straightforward manner through reducing $A$ and $f(A)$ to the diagonal forms. Introducing the matrices 
\begin{gather}
S=
\left[\begin{smallmatrix}
1 & 0 & 0\\
\tfrac{g_1}{a_2-a_1} & 1 & 0\\[3pt]
\tfrac{g_1g_2}{(a_2-a_1)(a_3-a_1)} & \tfrac{g_2}{a_3-a_2} & 1
\end{smallmatrix}\right]\label{eq:S}\\
S^{-1}=
\left[
\begin{smallmatrix}
1 & 0 & 0\\
\tfrac{g_1}{a_1-a_2} & 1 & 0\\[3pt]
\tfrac{g_1g_2}{(a_1-a_3)(a_2-a_3)} & \tfrac{g_2}{a_2-a_3} & 1
\end{smallmatrix}\right],\label{eq:S-inv}
\end{gather}
a straightforward computation shows that
\begin{gather}
A=S\left[\begin{smallmatrix}
-a_1 & 0 & 0\\
0 & -a_2 & 0\\
0 & 0 & -a_3
\end{smallmatrix}\right]S^{-1},\label{eq:a-diagonal}\\
f(A)=S\left[\begin{smallmatrix}
f(-a_1) & 0 & 0\\
0 & f(-a_2) & 0\\
0 & 0 & f(-a_3)
\end{smallmatrix}\right]S^{-1}\label{eq:f-matr}.
\end{gather}
Equations~\eqref{eq:S}-\eqref{eq:f-matr} will be used in the proof of Theorem~\ref{th:stability}.

An important property of the divided differences is the generalized mean-value formula~\cite{DeBoor2005}. Assuming, without loss of generality, that $x_0<x_1<x_2$, one has
\begin{equation}\label{eq:mean-value}
\begin{gathered}
f[x_0,x_1]=f'(\bar\theta)\quad\text{for some $\bar\theta\in [x_0,x_1]$},\\
f[x_0,x_1,x_2]=\frac{1}{2}f''(\tilde\theta)\quad\text{for some $\tilde\theta\in [x_0,x_2]$.} 
\end{gathered}
\end{equation}
This entails the following simple yet important proposition.
\begin{proposition}\label{prop:convexity}
Let function $f:(-\infty,0)\to\mathbb{R}$ be strictly convex and $T>0$. Then, $Cf(TA)B>0$.
If, additionally, $f$ is increasing and positive, then the vector $f(TA)B$ is positive.
\end{proposition}
\begin{IEEEproof}
The proof is straightforward by applying the Opitz formula and~\eqref{eq:mean-value} to the matrix $TA$ with the eigenvalues $(-Ta_i)<0$ and noticing that
 $Cf(TA)B$ is nothing else than the $(3,1)$ entry of $f(TA)$ and $f(TA)B$ is its first column.
\end{IEEEproof}
\begin{corollary}\label{cor.mu-convex}
For $\mu(x)=1/(\e^{-x}-1)$ and $\nu(x)=-x\mu(x)$, the vectors $\mu(TA)B$ and $\nu(TA)B$ are positive. 
\end{corollary}
\begin{IEEEproof}
Obviously, the function $\mu(x)=1/(\e^{-x}-1)$ restricted to the interval $x\in(-\infty,0)$ is positive and increasing.
After some computation, one checks that
\[
\mu''(x)=\frac{\e^{-x}(1+\e^{-x})}{(\e^{-x}-1)^3}>0\quad\forall x<0,
\]
in other words, $\mu$ is strictly convex on $(-\infty,0)$. The first statement is now immediate from Proposition~\ref{prop:convexity}.

Similarly, $\nu(x)>0$ for $x<0$. To prove the monotonicity and convexity, note first that
$1-(x+1)\e^{-x}$ is strict decreasing on $(-\infty,0]$ and thus achieves its minimum ($0$) at $x=0$.
Therefore, $\nu$ is strict increasing on $(-\infty,0]$ as
\[
\nu'(x)=\frac{1-(x+1)\e^{-x}}{(1-\e^{-x})^2}>0\quad\forall x<0.
\]
Finally, a direct yet tedious computation shows that
\[
\nu''(x)=-\e^{-x}\frac{(2+x)\e^{-x}+x-2}{(\e^{-x}-1)^3},
\]
where the numerator is negative for $x<0$ as
\[
\frac{d}{dx}\left((2+x)\e^{-x}+x-2\right)=1-(1+x)\e^{-x}>0\quad\forall x<0.
\]
Hence, the function $(2+x)\e^{-x}+x-2$ is increasing on $(-\infty,0)$ and achieves its strict maximum ($0$) at $x=0$.
This proves the strict convexity $\nu''(x)>0\;\;\forall x<0$.
\end{IEEEproof}
\begin{corollary}\label{cor.rho-convex}
For $\varrho(x)=x/(1-\e^x)$, the inequality $C\varrho(TA)B<0$ holds for all $T>0$. 
\end{corollary}
\begin{IEEEproof}
Similar to the previous corollary, one obtains
\[
\varrho''(x)=\frac{\e^x}{(1-\e^x)^3}\left(x+2+x\e^x-2\e^x\right),
\]
As we know from the previous proof, $(x+2)\e^{-x}+x-2<0$ when $x<0$. 
Thus, $-\varrho$ is strictly convex on $(-\infty,0)$. Applying Proposition~\ref{prop:convexity} to $-\varrho$ yields
$(-C\varrho(TA)B)>0$.
\end{IEEEproof}

\section{Proofs of Theorem~\ref{pro:fp} and Lemma~\ref{lem.signs_DJ}}\label{app.thm1}

Applying the Opitz formula (Appendix~\ref{app.opitz}) to the matrix $TA$, where $A$ is defined as in~\eqref{eq:1}, $T>0$, and $f\equiv\mu$, one easily checks that
~\eqref{eq:fp_alpha} can be equivalently rewritten\footnote{The inverse matrix is well-defined since the eigenvalues $\e^{Ta_i}$ of $\e^{-TA}$ are, obviously, greater than $1$.} as
\begin{equation*}%\label{eq:fp_alpha+}
X=\lambda\mu(TA)B=\lambda(\e^{-AT}-I)^{-1}B,    
\end{equation*}
which, in turn, is equivalent to the relation
\begin{equation}\label{eq:aux1}
X=\e^{TA}(X+\lambda B).%\Longleftrightarrow X=\lambda (\e^{-TA}-I)^{-1}B=\lambda\mu(TA)B.
\end{equation}
In view of Corollary~\ref{cor.mu-convex}, the vector $X$ from~\eqref{eq:fp_alpha} is positive.

\subsubsection*{\bf Proof of Theorem~\ref{pro:fp}}
To prove the implication $2)\Longrightarrow 1)$, recall that 1-cycle with the impulse weight $\lambda_n\equiv\lambda$ and
the pulse width $t_{n+1}-t_n\equiv T$ corresponds to the fixed point $X(t_n^-)\equiv X$ of $Q$ whose output
$\bar y_0=CX$ satisfies~\eqref{eq:f-phi-correspondence}.
Substituting $X_n=X_{n+1}=X$ and~\eqref{eq:f-phi-correspondence} into~\eqref{eq:map}, one proves~\eqref{eq:aux1}, which is equivalent to~\eqref{eq:fp_alpha}.

The proof of implication $1)\Longrightarrow 2)$ is similar: Since~\eqref{eq:fp_alpha} and~\eqref{eq:f-phi-correspondence} (where $\bar y_0=CX$) imply~\eqref{eq:aux1}, the vector $X$ from~\eqref{eq:fp_alpha} is a fixed point of $Q$. Furthermore,~\eqref{eq:f-phi-correspondence} implies that $X$ corresponds to a 1-cycle with parameters $\lambda,T$.\hfill$\blacksquare$

\subsubsection*{\bf Proof of Lemma~\ref{lem.signs_DJ}}\label{app.lemma-proof} %from Arxiv file

The proof follows from the definition of vectors $D,J$ and the Opitz formula in Appendix~\ref{app.opitz}.

The vector $J$, being the first column of the matrix $\e^{AT}$, is strictly positive in view of Proposition~\ref{prop:convexity} 
as the function $x\mapsto \e^x$ is monotone increasing and convex on the real line.

In order to prove that $D=AX<0$, notice that
\[
D=\lambda A(\e^{-AT}-I)^{-1}B=-\lambda T^{-1}\nu(TA)B,
\] 
where $\nu(z)=\frac{z}{1-\e^{-z}}$. Hence, $D<0$ due to Corollary~\ref{cor.mu-convex}.
\hfill$\blacksquare$

\section{Proof of Theorem~\ref{th:stability}}\label{app.thm2}

First,  some technical properties for the functions $\chi,\psi$ defined in~\eqref{eq.chi} and~\eqref{eq.psi} are established. Also, unless otherwise stated,  $\xi\leq 0$ and $\eta\geq 0$ assumed. Recall that, in view of~\eqref{eq:1+}, one has $a_1<a_2<a_3$.
Along with the functions~\eqref{eq.psi} and~\eqref{eq.c0}, we will use the characteristic polynomial of $\mathscr{Q}$, that is,
\begin{equation}
\chi(s|\xi,\eta)\triangleq\det\left(sI-\mathscr{Q}(\xi,\eta)\right).\label{eq.chi}
\end{equation}

We first prove several technical propositions.
\begin{proposition}\label{prop.positive-eigs}
The functions $\psi(s|\xi,\eta)$ and $\chi(s|\xi,\eta)$ have no zeros in the interval $s\in(\e^{-a_1T},\infty)$. In particular, all \emph{real positive} eigenvalues of matrix $\mathscr{Q}$, if they exist, are stable.
\end{proposition}
\begin{IEEEproof}
Retracing the proof of Lemma~\ref{lem.signs_DJ}, recall that $J>0$ and $D<0$, therefore, $\xi J\leq 0$ and $\eta D\leq 0$. 
Since $(-a_1)$ is the maximal eigenvalue of $A$ thanks to $0<a_1<a_2<a_3$, $\e^{-a_1T}$ is the spectral radius of $\e^{TA}$. Hence, 
\[
(sI-\e^{AT})^{-1}=s^{-1}(I-s^{-1}\e^{AT})^{-1}=\sum\nolimits_{k=0}^{\infty}s^{-k-1}\e^{kTA},
\]
is a nonnegative matrix. In view of~\eqref{eq.psi}, this entails that $\psi(s)\geq 1$ for $s>\e^{-a_1T}$.
Moreover, $\det(sI-\e^{AT})>0$ for $s>\e^{-a_1T}$, and hence $\chi(s)>0$ on this interval.
\end{IEEEproof}

\begin{proposition}\label{prop.residue}
The residue of $\psi(s)=\psi(s|\xi,\eta)$ at the pole $s=\e^{-a_3T}$ is nonnegative:
$
\lim_{s\to\e^{-a_3T}}\left(s-\e^{-a_3T}\right)\psi(s)\geq 0.
$
If $(\xi,\eta)\ne(0,0)$, then this residue is positive, and hence
$\psi(s)\to -\infty$ as $s\to \e^{-a_3T}-0$.
\end{proposition}
\begin{IEEEproof}
Applying~\eqref{eq:f-matr} to $f(z)=f_s(z)$, where
\[
f_s(z)\triangleq(sI-\e^{Tz})^{-1}\left(\e^{Tz}\xi+z(\e^{-Tz}-1)^{-1}\eta\lambda\right), 
\]
and using the definition of $J,D$, one proves that
\[
\psi(s)\overset{\eqref{eq.psi}}{=}1-Cf_s(A)B=1-\sum\nolimits_{i=1}^3\bar c_i\bar b_if_s(-a_i),
\]
where $\bar b_i,\bar c_i$ are the coordinates of the vectors $\bar B:=S^{-1}B$, $\bar C:=CS$, and $S$ is the matrix from~\eqref{eq:S}. 
Obviously, $f_s(-a_i)$ exists and is analytic (in $s$) in a vicinity of $s=\e^{-a_3T}$ for $i=1,2$, and hence
the residue in question is found as
\[
\begin{aligned}
\lim_{s\to \e^{-a_{3}T}}&\big(s-\e^{-a_{3}T}\big)\psi(s)=\\&=-\bar c_3\bar b_3\lim_{s\to \e^{-a_{3}T}}\big(s-\e^{-a_{3}T}\big)f_s(-a_3)=\\
&=-\bar c_3\bar b_3\left(\xi \e^{-a_3T}-\eta\lambda a_3(\e^{a_3T}-1)^{-1}\right).
\end{aligned}
\]
Using~\eqref{eq:S} and~\eqref{eq:S-inv}, one checks that $\bar c_3=1$ and $\bar b_3=g_1g_2/(a_2-a_3)(a_1-a_3)>0$. Hence the residue is nonnegative for all $\xi\leq 0$ and $\eta\geq 0$, being positive unless $\xi=\eta=0$. 
\end{IEEEproof}

\begin{proposition}\label{prop.c0}
The function $c(\eta)$ defined by~\eqref{eq.c0} is strictly decreasing, in particular, $c(\eta)\leq c(0)=\e^{-(a_1+a_2+a_3)T}$ for $\eta\geq 0$.
Hence, $\psi(c(\eta)|\xi,\eta)$ is well-defined for $\eta\geq 0$.
\end{proposition}
\begin{IEEEproof}
The first statement follows from Corollary~\ref{cor.rho-convex}, entailing that $T^{-1}C\varrho(TA)B=CA(I-\e^{AT})^{-1}B<0$, and using~\eqref{eq.c0}.
The second statement is straightforward by recalling that the poles of $\psi$ are $\e^{-a_iT}>c(\eta)$.
\end{IEEEproof}

\begin{proposition}\label{prop.negative-eigs}
If~\eqref{eq.necess-stab2} holds, then the number of roots of the polynomial $\chi(s)$ on the interval $(-\infty,-1]$, counted with multiplicity, is $0$ or $2$.
\end{proposition}
\begin{IEEEproof}
Since $\chi(s) \to -\infty$ as $s \to -\infty$ and $\chi(-1) = -\psi(-1)(1 + \e^{-a_1T})(1 + \e^{-a_2T})(1 + \e^{-a_3T}) < 0$ if~\eqref{eq.necess-stab2} holds, the number of real roots of $\chi$ on $(-\infty, -1]$, counted with multiplicity, must be even. Consequently, the cubic polynomial $\chi$ has either no roots or two roots in this interval.
\end{IEEEproof}

\subsection*{Necessity of~\eqref{eq.necess-stab1} and~\eqref{eq.necess-stab2}}

%We start the proof of Theorem~\ref{th:stability} by establishing the following lemma.
\begin{lemma}\label{lem.tech1}
If matrix $\mathscr{Q}$ is Schur stable, then the conditions~\eqref{eq.necess-stab1} and~\eqref{eq.necess-stab2} hold.
\end{lemma}
\begin{IEEEproof}
To prove~\eqref{eq.necess-stab1}, notice that $C\e^{-AT}J=CB=0$ and  $C\e^{-AT}D=\lambda CA(I-\e^{AT})^{-1}B$ in view of~\eqref{eq:Q_prime_affine}.
Hence, $\psi(0)=1+\eta\lambda CA(I-\e^{AT})^{-1}B$. 
On the other hand, $\det\mathscr{Q}=-\chi(0)=\det\e^{AT}\psi(0)$ is the product of the eigenvalues of $\mathscr{Q}$. The Schur stability entails~\eqref{eq.necess-stab1}, because $|\chi(0)|<1$ and
\[
\psi(0)=-\det\e^{-AT}\chi(0)>-\det\e^{-AT}=-\e^{(a_1+a_2+a_3)T}.
\]

Inequality~\eqref{eq.necess-stab2} is obtained by noticing that, on one hand, $\det(I+\e^{AT})=(1+\e^{-a_1T})(1+\e^{-a_2T})(1+\e^{-a_3T})>0$, and, on the other hand, 
$\chi(-1)<0$, since $\chi(s)\to-\infty$ as $s\to-\infty$ and $\mathscr{Q}$ does not have eigenvalues on $(-\infty, -1]$. Hence, $\psi(-1)=-\chi(-1)/\det(I+\e^{AT})>0$.
\end{IEEEproof}

\subsection*{Theorem~\ref{th:stability}: Non-Critical Case}

By virtue of Proposition~\ref{prop.c0}, one has $c(\eta)\in (-\infty,\e^{-a_3T})$, on which interval the function $\psi(s|\xi,\eta)$ is well-defined. Recall that $
\psi(0|\xi,\eta)=\e^{(a_1+a_2+a_3)T}c(\eta)$ and, in view of the definition of $\psi$ in~\eqref{eq.psi}, one has 
\[
\begin{aligned}
\det\mathscr{Q}(\xi,\eta)=-\chi(0|\xi,\eta)
=\psi(0|\xi,\eta)\e^{-(a_1+a_2+a_3)T}=c(\eta).
\end{aligned}
\]

\emph{Necessity part}. The necessity of~\eqref{eq.necess-stab1} and~\eqref{eq.necess-stab2} is implied by  Lemma~\ref{lem.tech1}. To prove the necessity of~\eqref{eq.necess-stab3}, notice that its violation $\psi(c(\eta))c(\eta)\leq 0$ implies that $\psi(c(\eta))\psi(0)\leq 0$, that is,
$\psi(s)$ and $\chi(s)$ have a root $s_*$ (i.e., an eigenvalue of $\mathscr{Q}$) lying between $0$ and $c(\eta)$.
Since $\psi(0)\ne 0$, one has $0<|s_*|\leq |c(\eta)|=|\det\mathscr{Q}|$, entailing that the product of two other eigenvalues of $\mathscr{Q}$ should be at least $1$ in modulus. Hence, violation of~\eqref{eq.necess-stab3} is incompatible with the Schur stability of $\mathscr{Q}$. 

\emph{Sufficiency part}. 
The case $\xi=\eta=0$ is trivial ($\mathscr{Q}(0,0)=\e^{AT}$ is Schur stable). Assume now that~\eqref{eq.necess-stab1},~\eqref{eq.necess-stab2}, and~\eqref{eq.necess-stab3} hold, and $(\xi,\eta)\ne(0,0)$. It will be proven that $\mathscr{Q}$ has
one real eigenvalue $s_1\in(-1,\e^{-a_3T})$, whereas two other of its eigenvalues $s_2,s_3\in\mathbb{C}$ satisfy the condition $0<s_2s_3<1$. Recall that $c(\eta)<\e^{-a_3T}$ in view of Proposition~\ref{prop.c0}, and $\psi(s)$ is continuous for $s\in(-\infty,\e^{-a_3T})$.

Indeed, if $c(\eta)>0$, one has $\psi(c(\eta))>0$ and, by virtue of Proposition~\ref{prop.residue}, $\psi(s)$ has a root $s_1\in(c(\eta),\e^{-a_3T})$, being an eigenvalue of 
$\mathscr{Q}$. On the other hand, if $c(\eta)<0$, then~\eqref{eq.necess-stab1} entails that $c(\eta)>-1$. Due to~\eqref{eq.necess-stab2} and~\eqref{eq.necess-stab3}, entailing that $\psi(-1)>0>\psi(c(\eta))$, the matrix $\mathscr{Q}$ has an eigenvalue
$s_1\in(-1,c(\eta))$. In both situations, $s_1$ has the same sign as $c=\det\mathscr{Q}=s_1s_2s_3$ and $|s_1|>|c|$.
Hence, $0<s_2s_3<1$.

If we have a pair of complex-conjugate roots 
$s_2=s_3^*$, then $|s_2|=|s_3|=\sqrt{|s_2s_3|}<1$. If $s_2,s_3>0$, then $s_2,s_3<\e^{-a_1T}$ thanks to Proposition~\ref{prop.positive-eigs}. Finally, if $s_2,s_3<0$, then
$s_2,s_3>-1$ due to Proposition~\ref{prop.negative-eigs} (since, otherwise, $s_2,s_3\leq -1$ and $s_2s_3\geq 1$). Thus, in all possible situations, $\mathscr{Q}$ is Schur stable, having three eigenvalues inside the unit disk.
We have also proved that the spectral radius of $\mathscr{Q}$ is not less than $|s_1|>|c(\eta)|$. The proof of Theorem~\ref{th:stability} in the non-critical
case is complete.

\subsection*{Theorem~\ref{th:stability}: Critical Case}

Notice first that $c(\eta)=0$ if and only if 
\[
\eta=\eta_*\triangleq -\frac{1}{\lambda CA(I-\e^{TA})^{-1}B},
\]
where the denominator is positive (Proposition~\ref{prop.c0}). As has been shown, $\psi(0|\xi,\eta_*)=\e^{(a_1+a_2+a_3)T}c(\eta_*)=0$,
that is, the matrix $\mathscr{Q}$ has eigenvalue at $0$.
Introducing the derivative
\[
\psi'(s|\xi,\eta)\triangleq\frac{\partial\psi}{\partial s}(s|\xi,\eta)=C(sI-\e^{AT})^{-2}(\xi J+\eta D),
\]
the condition~\eqref{eq.necess-stab3+} can be written as $|\psi'(0|\xi,\eta)|<\e^{(a_1+a_2+a_3)T}$.
To prove Theorem~\ref{th:stability},  consider two cases.

\textbf{Case 1:} The matrix $\mathscr{Q}$ has a multiple eigenvalue at $0$, that is, for a given triple $T>0$, $\xi\leq 0$, and $\eta=\eta_*$, it holds that $\psi(0)=\psi'(0)=0$. In this case,~\eqref{eq.necess-stab3+} is fulfilled automatically. It was shown that~\eqref{eq.necess-stab2} is necessary for Schur stability; if~\eqref{eq.necess-stab2} holds, the third eigenvalue of $\mathscr{Q}$ (which is automatically real) lies between $-1$ and $\e^{-a_1T}$ due to Propositions~\ref{prop.positive-eigs} and~\ref{prop.negative-eigs}.

\textbf{Case 2:} Assume now that only one of the three eigenvalues $s_1(\eta),s_2(\eta),s_3(\eta)$ vanishes at
$\eta=\eta_*$, without loss of generality, assume that $s_1(\eta_*)=0$. Since this eigenvalue is simple, one has $\psi'(0|\xi,\eta_*)\ne 0$.
In view of the implicit function theorem, $s_1(\eta)$ is $C^1$-smooth in a vicinity of $\eta=\eta_*$. Differentiating the relation
$
\psi(s_1(\eta)|\xi,\eta)=0
$
with respect to $\eta$ and substituting $s_1(\eta_*)=0$, one shows that
\[
\begin{aligned}
\frac{ds_1(\eta_*)}{d\eta}&\psi'(0|\xi,\eta_*)=-\frac{\partial\psi}{\partial\eta}(0|\xi,\eta)\overset{\eqref{eq.psi}}{=}-C\e^{-AT}D=\\
&=-\lambda CA(I-\e^{TA})^{-1}B\overset{\eqref{eq.c0}}{=}-\e^{(a_1+a_2+a_3)T}\frac{dc(\eta_*)}{d\eta}.
\end{aligned}
\]
Recalling that $\det\mathscr{Q}(\xi,\eta)=c(\eta)$, the L'H\^opital rule leads to
\[
\begin{aligned}
\lim\nolimits_{\eta\to\eta_*}s_2(\eta)s_3(\eta)&=\lim\nolimits_{\eta\to\eta_*}\frac{c(\eta)}{s_1(\eta)}=\frac{dc(\eta_*)/d\eta}{ds_1(\eta_*)/d\eta}\\
&=-\e^{-(a_1+a_2+a_3)T}\psi'(0|\xi,\eta_*).
\end{aligned}
\]
Hence,~\eqref{eq.necess-stab3+} is equivalent to the inequality $|s_2(\eta_*)s_3(\eta_*)|<1$, which is necessary for the Schur stability 
of $\mathscr{Q}(\xi,\eta_*)$, whereas the necessity of~\eqref{eq.necess-stab2} is ensured by Lemma~\ref{lem.tech1}.

On the other hand, if $|s_2(\eta_*)s_3(\eta_*)|<1$ and~\eqref{eq.necess-stab2} hold, then, similar to the proof of non-critical case,
one shows that either $s_2(\eta_*)=\bar s_3(\eta_*)$ (and their modulus is thus less than $1$) or $s_2(\eta_*),s_3(\eta_*)\in (-1,\e^{-a_1T})$.
Since $s_1(\eta_*)=0$, the matrix $\mathscr{Q}$ is Schur stable. This proves the sufficiency part.

\section{The Jacobian in the Amplitude Modulation Case}\label{app:E}

Applying the Opitz formula to $f(x)=\e^x$, the Jacobian matrix~\eqref{eq:jacobian-amplitude} can be found explicitly, 
yielding the following.

%The vector $J=\e^{TA}B$ is the first column of the matrix exponential whose entries are (see Appendix~\ref{app.opitz})
%\begin{align}\label{eq:opitz_exp}
%&\e^{At}=\\
%&\left[\begin{array}{ccc}
%    \e^{-a_1t} &0 &0\\
%    g_1t \e \lbrack -a_1t, -a_2t \rbrack &\e^{-a_2t} &0\\
%    g_1g_2t^2 \e \lbrack -a_1t, -a_2t,  -a_3t\rbrack & g_2t\e \lbrack -a_2t, -a_3t \rbrack & \e^{-a_3t}
%\end{array}\right]. \nonumber
%\end{align}
%The following proposition is thus straightforward.
\begin{proposition}\label{pr:amplitude_modulation}
The characteristic polynomial of $Q^\prime_F(X)$ is
%\begin{equation}\label{eq:char_poly_F}
 $\chi(s)= s^3-\gamma_1 s^2-\gamma_2 s -\gamma_3$,
%\end{equation}
where $\gamma_3= \e^{-(a_1+a_2+a_3)T}$ and
\begin{align*}
    \gamma_1&= \e^{-a_1T}+\e^{-a_2T}+\e^{-a_3T}\\
    &+F^\prime(\bar y_0)g_1g_2T^2 \e \lbrack -a_1T, -a_2T,  -a_3T\rbrack,\\
    \gamma_2&=F^\prime(\bar y_0) g_1g_2 T^2\left( \e \lbrack -a_1T, -a_2T \rbrack \e \lbrack -a_2T, -a_3T \rbrack\right.\\
      &- \left. \e^{-a_2T}\e \lbrack -a_1T, -a_2T,  -a_3T\rbrack \right)-\e^{-a_1T}(\e^{-a_2T}+\e^{-a_3T})\\
      &-\e^{-(a_2+a_3)T}. %\\
    %\gamma_3&= \e^{-(a_1+a_2+a_3)T}.
\end{align*}
\end{proposition}

The polynomial has a multiple root if and only if its discriminant vanishes~\cite{CoxLittleOShea2005}, that is,
$
\gamma_1^2\gamma_2^2
+4\gamma_2^3
-4\gamma_1^3\gamma_3
-27\gamma_3^2
-18\gamma_1\gamma_2\gamma_3
=0
$. Thanks to Proposition~\ref{pr:amplitude_modulation}, $\gamma_1$ and $\gamma_2$ are affine functions of $F'(\bar y_0)$, whereas $\gamma_3$ is independent of $F'(\bar y_0)$. Hence, the discriminant condition yields a fourth-order equation in $F'(\bar y_0)$, which can be solved numerically.

Notice that at the double multiplier point, where two eigenvalues are equal $s_1=s_2$, one has
%\begin{equation*}%\label{eq:Hopf}
    $\chi(s)= {(s-s_1)}^2(s-s_3)$ and $\chi(s_1)=\chi'(s_1)=0$.
%\end{equation*}
Notice that $\chi'(r)=0$ is a quadratic equation, having two roots
$
r_\pm=(\gamma_1\pm\sqrt{\gamma_1^2+3\gamma_2})/3.
$
The double multiplier is the value $s_1=r_\pm$ for which $\chi(s_1)=0$.
The remaining multiplier is then found as
$s_3=\gamma_3/s_1^2$, and the spectral radius of the Jacobian is $\max (|s_1|,|s_3|)$.

%The roots of the characteristic polynomial in this case are
%\[
%\xi_1=-\frac{\gamma_1}{2}=\sqrt{\gamma_2}, \quad \xi_3=-\frac{\gamma_3}{\gamma_2}.
%\]
%The achievable spectral radius of the Jacobian is $\max (|\xi_1|,|\xi_3|)$. 
%The slope tangent $F^\prime(\bar y_0)$ at the double multiplier point, in turn, can be found from the relation between $\gamma_1$ and $\gamma_2$.

%\bf{If you include a photo:}\vspace{-33pt}
%\begin{IEEEbiography}[{\includegraphics[width=1in,height=1.25in,clip,keepaspectratio]{fig1}}]{Michael Shell}
%Use $\backslash${\tt{begin\{IEEEbiography\}}} and then for the 1st argument use $\backslash${\tt{includegraphics}} to declare and link the author %photo.
%Use the author name as the 3rd argument followed by the biography text.
%\end{IEEEbiography}



\vfill

\end{document}